\documentclass[11pt,a4paper]{amsart}
\usepackage[margin=2.5cm,marginpar=2cm]{geometry}

\usepackage{amscd,latexsym,amsthm,amsfonts,amssymb,amsmath,amsxtra,eucal}
\usepackage[all]{xy}

\usepackage[bookmarks,colorlinks]{hyperref}
\usepackage{cleveref}
\usepackage{graphicx}

\usepackage{verbatim}
\usepackage{mathabx}
\usepackage{mathrsfs}
\usepackage{bbm}

\usepackage{ulem}
\usepackage{enumitem}

\usepackage[scr=boondoxo]{mathalfa}

\def\sfs{\mathscr{s}}

\renewcommand\theequation{\thesection.\arabic{equation}}

\newcommand{\BA}{{\mathbb {A}}}

\newcommand{\BC}{{\mathbb {C}}}

\newcommand{\BF}{{\mathbb {F}}}

\newcommand{\BH}{{\mathbb {H}}}

\newcommand{\BN}{{\mathbb {N}}}

\newcommand{\BQ}{{\mathbb {Q}}}

\newcommand{\BZ}{{\mathbb {Z}}}

\newcommand{\CB}{{\mathcal {B}}}

\renewcommand{\CD}{{\mathcal {D}}}
\newcommand{\CE}{{\mathcal {E}}}
\newcommand{\CF}{{\mathcal {F}}}

\newcommand{\CJ}{{\mathcal {J}}}
\newcommand{\CK}{{\mathcal {K}}}
\newcommand{\CL}{{\mathcal {L}}}

\newcommand{\CO}{{\mathcal {O}}}
\newcommand{\CP}{{\mathcal {P}}}

\newcommand{\CS}{{\mathcal {S}}}
\newcommand{\CT}{{\mathcal {T}}}
\newcommand{\CU}{{\mathcal {U}}}

\newcommand{\Fe}{{\mathfrak {e}}}
\newcommand{\Ff}{{\mathfrak {f}}}

\newcommand{\Fp}{{\mathfrak {p}}}

\newcommand{\RE}{{\mathrm {E}}}
\newcommand{\RF}{{\mathrm {F}}}
\newcommand{\RG}{{\mathrm {G}}}

\newcommand{\RU}{{\mathrm {U}}}

\newcommand{\Ad}{{\mathrm{Ad}}}

\newcommand{\Asai}{{\mathrm{As}}}

\newcommand{\End}{{\mathrm{End}}}

\newcommand{\Gal}{{\mathrm{Gal}}}
\newcommand{\GL}{{\mathrm{GL}}}

\newcommand{\Hom}{{\mathrm{Hom}}}

\newcommand{\Ind}{{\mathrm{Ind}}}

\newcommand{\Ker}{{\mathrm{Ker}}}

\newcommand{\reg}{{\mathrm{reg}}}

\newcommand{\SL}{{\mathrm{SL}}}

\newcommand{\Sp}{{\mathrm{Sp}}}

\def\bartau{\bar{\tau}}

\def\sfSp{{\sf Sp}}

\newtheorem{thm}{Theorem}[section]
\newtheorem{cor}[thm]{Corollary}
\newtheorem{lem}[thm]{Lemma}
\newtheorem{prop}[thm]{Proposition}
\newtheorem {conj}[thm]{Conjecture}
\newtheorem {ques/conj}[thm]{Question/Conjecture}

\newtheorem{defn}[thm]{Definition}
\newtheorem{rmk}[thm]{Remark}

\providecommand{\iddots}{\scalebox{-1}[1]{$\ddots$}}
\providecommand{\iddots}{\rotatebox[c]{180}{$\ddots$}}
\usepackage[numbers]{natbib}

\newcommand{\Irr}{{\rm Irr}}

\let\ssout\sout
\def\sout#1{{\color{red}\ssout{#1}}}

\usepackage{braket}
\def\cc{c}

\def\bF{\mathbb{F}}
\DeclareMathOperator{\cInd}{c-Ind}

\def\half{\frac{1}{2}}
\def\barsigma{\bar{\sigma}}
\def\bC{{\mathbb C}}
\def\sfH{{\sf H}}
\def\sfM{{\sf M}}
\def\sfS{{\sf S}}
\def\sfT{{\sf T}}
\def\Femo{\Fe_m^{\circ}}
\def\Femoo{\Fe_m^{\circ,1}}

\def\SCEFm{\mathrm{SC}_0(\GL_m(\RE)/\GL_m(\RF))}

\def\DD{{\mathscr{D}}}

\def\inn#1#2{\left\langle
      \def\ta{#1}\def\tb{#2}
      \ifx\ta\@empty{\;} \else {\ta}\fi ,
      \ifx\tb\@empty{\;} \else {\tb}\fi
      \right\rangle}

\begin{document}
\renewcommand{\theequation}{\arabic{equation}}
\numberwithin{equation}{section}

\title[Fourier-Jacobi local descent]{Fourier-Jacobi models of Deligne-Lusztig characters and depth zero local descent for unitary groups}

\date{\today}

\author[Dongwen Liu]{Dongwen Liu}

\address{School of Mathematical Sciences, Zhejiang University, Hangzhou 310058, Zhejiang, P.R. China}

\email{maliu@zju.edu.cn}

\author[Jiajun Ma]{Jiajun Ma}

\address{School of Mathematical Sciences, Xiamen University, Xiamen, P. R. China}
\address{Department of Mathematics, Xiamen University Malaysia Campus, Malaysia}

\email{jiajun.ma@xmu.edu.my}

\long\def\delete#1{\relax}

\author[Fang Shi]{Fang Shi}

\address{School of Mathematical Sciences, Zhejiang University, Hangzhou 310058, Zhejiang, P.R. China}

\email{11935007@zju.edu.cn}


\begin{abstract}
In this paper, we deduce explicit multiplicity formulas of the Fourier-Jacobi model for Deligne-Lusztig characters of finite symplectic groups, unitary groups, and general linear groups.
We then apply these results to 
deduce the explicit depth zero local descent (\`a la Soudry and Tanay) for $p$-adic unitary groups.
The result is a concrete example in the context of non-tempered Gan-Gross-Prasad program. 
\end{abstract}

\maketitle

\section{Introduction}

This paper was motivated by the celebrated Gan-Gross-Prasad conjectures (\cite{GGP1, GGP2}) in the Fourier-Jacobi case, which study certain branching laws for representations of classical groups over finite fields, local fields, or adele rings of global fields. We  have  a two-fold purpose in this paper.

We first determine the multiplicity of the Fourier-Jacobi model for two Deligne-Lusztig characters of a finite symplectic, unitary, or general linear group.
The Bessel case was settled by Reeder \cite{R}, who in fact made a qualitative study for the restrictions of Deligne-Lusztig characters in a general setting. Since the characters of Weil representations are not uniform functions, Reeder's approach does not apply to our situation directly. The new input is the geometrization of the character of Weil representations \cite{GH}. As a consequence, the multiplicity formula has to be viewed as a function of {\it geometric type} (see \Cref{geotype}). 


Secondly,  we study the local descent of an irreducible depth zero distinguished supercuspidal representation of $\GL_m(\RE)$ to the quasi-split unitary group $\RU_{2n}(\RF)$ in the sense of \cite{ST}, where 
$\RE/\RF$ is a quadratic extension of $p$-adic local fields, and $n:=\lfloor \frac{m+1}{2}\rfloor$. This local descent problem is a special case of the non-tempered local Gan-Gross-Prasad conjecture proposed in \cite{GGP2}.
The non-tempered local GGP conjecture is recently proved for general linear groups over $p$-adic local fields in \cite{Ch} and remains open in general. 
Supercuspidal types encode arithmetic data, rendering them more amenable in number-theoretical applications. Their relationship with the local Langlands correspondence has been deciphered in many cases \cite{BH, Ka}, covering all representations pertinent to our setting.

\subsection{Overview of Reeder's work and motivation}
Let $G$ be a connected simple algebric group and $H$ be a closed reductive subgroup of $G$ defined over $\Ff := \bF_q$.  
Let $F$ denote the Frobenius automorphism so that $G^F$ is the $\Ff$-points of $G$. 
Reeder's method \cite{R} for computing the pairing 
\[
\langle R^G_{T,\chi}, R^H_{S,\eta} \rangle_{H^F}
\]
of two Deligne-Luszting characters (as class functions) can be outlined in the following steps:
\begin{enumerate}[label=\arabic*., wide]
\item Interpret $R^G_{T,\chi}$ and $R^H_{S,\eta}$ as  sequences of the Deligne-Lusztig characters of $G^{F^\nu}$ and $H^{F^\nu}$ respectively, where $\nu$ varies in a certain arithmetic progression.
\item Demonstrate that 
the pairing  
$ \langle R^G_{T,\chi}, R^H_{S,\eta} \rangle_{H^{F^\nu}}$
is a polynomial function $M$ in $q^\nu$, see \Cref{sec:reeder}.
\item Show that the degree of $M$ measures the relative complexity between the groups $G$ and $H$, and the leading coefficient of $M$ can be calculated in terms of combinatorial data.
\end{enumerate}
Under specific conditions, the polynomial $M$ is reduced to a constant. This property facilitated Reeder's application of his method to the Gross-Prasad problem, thereby enabling him to derive a multiplicity formula for the Bessel models of finite special orthogonal groups.

Inspired by Reeder's work, we formulate a framework that incorporates the Weil representation in the character pairing and investigates the Fourier-Jacobi case of the Gan-Gross-Prasad conjecture. It is worth mentioning that this framework also applies to more generalized models of spherical varieties (as seen in \cite{SV}) over finite fields, which plays a crucial role in \cite{Shi}.
 
\delete{
Let us give a sketch of the approach in \cite{R}. The goal there was to compute the multiplicity 
\[
\langle R^G_{T,\chi}, R^H_{S,\eta} \rangle_{H^F}
\]
for two Deligne-Luszting characters in the classical situation of the Gross-Prasad conjecture, that is, $H\subset G$
are special orthogonal groups defined over a finite field $\mathbb{F}_q$ and $F$ is the Frobenius automorphism. Reeder developed a qualitative study in a general setting where $G$ is connected and simple, and $H$ is a closed connected
reductive subgroup. His method was to use the explicit character formula for Deligne-Lusztig characters, with $F$ replaced by its powers $F^\nu$ where $\nu$ varies in a certain arithmetic progression. The corresponding multiplicity 
$M(q^\nu)$ is a polynomial function of $q^\nu$ whose degree measures certain complexity for the pair $(G, H)$. A general complicated expression for the leading term was given. In the application for the Gross-Prasad situation, the complexity vanishes and the polynomial is a constant, which simplifies the leading term and gives a multiplicity formula for the Bessel models of finite special orthogonal groups. More details for the general multiplicity formula will be explained in Section \ref{sec:reeder}.

The general theory appears to be interesting in its own right. Similar to Reeder's work, in this paper we consider the Fourier-Jacobi situation of the Gan-Gross-Prasad conjecture. These works might provide a framework for the study of more general models of spherical varieties (see \cite{SV}) over finite fields.
}

\subsection{Fourier-Jacobi model for finite classical groups}
Let $\Ff$ be a finite field of characteristic $p>2$ and cardinality $q$, along with a fixed algebraic closure $\bar\Ff$.
For a subfield $\Fe$ contained in  $\bar\Ff$, designate $\Fe_d$ ($d\geq 1$) to be the degree $d$ extension of $\Fe$ in $\bar\Ff$. Let
\[
\Fe^\circ := \Set{a\in \Fe | \text{$a$ is not contained in a proper subfield of $\Fe$} }
\]
be the set of ``regular elements'' of $\Fe$. 
When $\Fe$ is an even degree extension of $\bF_p$, let 
\[
\Fe^1:= \set{a\in \Fe | a^{\sqrt{|\Fe|}+1}=1}
\] 
denote the set of ``norm one'' elements.

Let $G=\Sp_{2n}$, $\RU_n$ or $\GL_n$ be defined over $\Ff$ and let $F$ be the Frobenius endomorphism on $G$. 
For simplicity, we focus on the basic case of the Fourier-Jacobi model of a $G$-module (see \cite[Sections 12 and 19]{GGP1}) in this introduction.  The  general setting, for which the main reference is \cite[Section 12]{GGP1}, will be addressed in  Section \ref{ssec-UG} and Section 
\ref{ssec-LDGGP} for the finite field and $p$-adic field cases, respectively. 

In the basic case, we are interested in the pairing
\begin{equation}  \label{FJ-mult} 
m(\pi,\sigma) := \braket{\pi\otimes \omega_\psi^\vee, \sigma}_{G^F}.
\end{equation}
Here $\pi,\sigma$ are  (virtual) characters of $G^F$ and  $\omega_\psi$ denotes the Weil representation of $G^F$ attached to a non-trivial additive character $\psi$ of $\Ff$ in the sense of \cite{Ho1}. Note that  $m(\pi,\sigma)$ is the multiplicity of $\sigma$ occurring in $\pi \otimes \omega_\psi^\vee$ when $\pi$ and $\sigma$ are irreducible representations of $G^F$.
Our first main result, \Cref{MF}, evaluates the pairing \eqref{FJ-mult} for arbitrary Deligne-Lusztig characters (\cite{DL}) $\pi=R^G_{T,\chi}$ and $\sigma = R^G_{S,\eta}$,  
where $T$ and $S$ are $F$-stable maximal tori in $G$, and $\chi\in \Irr(T^F)$, $\eta\in \Irr(S^F)$. When $\chi$ and $\eta$ are {\it regular} in the sense that their 
stabilizers in the corresponding Weyl groups $W_G(T)^F$ and $W_G(S)^F$ are trivial, a straightforward multiplicity formula \eqref{reg-MF} is provided in \Cref{ssec-RC}  (c.f.  \cite[(9.9)]{R}).

\medskip 

Now let $G=\Sp_{2n}$ or $\RU_n$, and we evaluate \eqref{reg-MF} 
in a very special but important scenario.
Assume that $T, S$ are anisotropic and $\chi, \eta$ are regular. So $(-1)^{{\rm rk}\, G}R^G_{T,\chi}$ and $(-1)^ {{\rm rk} \, G}R^G_{S,\eta}$ are irreducible cuspidal representations of $G^F$, where ${\rm rk}\, G$ denotes the $\Ff$-rank of $G$ (see \cite[Theorem~8.3]{DL}). There are unique partitions $\lambda=(j^{\lambda_j})$ and $\mu=(j^{\mu_j})$ of $n$ such that
\[
T^F\cong \prod_j (\Ff_{2j}^1)^{\lambda_j}, \quad S^F\cong \prod_j (\Ff_{2j}^1)^{\mu_j}.
\]
The characters $\chi$ and $\eta$ decompose into  
\[
\chi =\bigotimes_j \chi_{j1}\otimes\cdots\otimes \chi_{j\lambda_j} \quad  \text{and} \quad \eta= \bigotimes_j \eta_{j1}\otimes\cdots\otimes \eta_{j\mu_j}
\]
accordingly with $\lambda_{jk}, \mu_{jl}\in \Irr(\Ff_{2j}^1)$. 

\begin{defn}
Let $\vartheta_j' \in \Irr(\Ff_{2j}^1)$ denote the unique nontrivial quadratic character. 
We say that $\chi$ and $\eta$ intertwine if $\chi_{jk}$ is a $\Gal(\bar\Ff/\Ff)$-conjugate of $\eta_{jl}\otimes \vartheta_j'$ for some $1\leq j\leq n$, $1\leq k\leq \lambda_j$ and $1\leq l\leq \mu_j$.
\end{defn}

Then we have the following result, which is in analogy with \cite[Theorem 1.2]{R}.

\begin{thm}
Suppose that $G=\Sp_{2n}$ or $\RU_n$, and that $T$ and $S$ are anisotropic $F$-stable maximal tori in $G$.
For regular characters $\chi\in \Irr(T^F)$, $\eta\in \Irr(S^F)$, we have
\[
\langle R^G_{T,\chi}\otimes \omega_\psi^\vee, R^G_{S,\eta}\rangle_{G^F} = \begin{cases}
0, & \textrm{if }\chi, \eta \textrm{ intertwine,} \\
1, & \textrm{otherwise.}
\end{cases}
\]
\end{thm}

\medskip
\subsection{Remarks on \Cref{MF}}
In contrast to the $p$-adic case where the multiplicity one theorem holds (\cite{Sun}),  the multiplicity $m(\pi, \sigma)$ for general $\pi, \sigma\in \Irr(G^F)$ could be greater than one in the finite field case by \Cref{MF} or \eqref{reg-MF}.
As an application of \Cref{MF}, for $G=\RU_n$ or $\GL_n$ we give a quick proof of a conjecture in \cite{HS} which asserts that $m(\pi,\sigma)$ is bounded by a function of $n$ that does not depend on $q=|\Ff|$. This is stated as \Cref{HS-A}, where we do not attempt to optimize the upper bound however. There has been a good deal of interest and progress in the topic of Weil representations of finite classical groups (see e.g. \cite{GHo, HS, HZ}), and the result given by \Cref{MF} should have applications towards the relevant problems. 


Let us sketch the proof of \Cref{MF}. Inspired by \cite{R}, we extend the pairing \eqref{FJ-mult} to the group $G^{F^\nu}$ of $\Ff_\nu$-points of $G$ and regard its value as a function $M(\nu)$ of $\nu$. As $\nu$ ranges over an arithmetic progression $\CP_m := \set{1+md : d\geq 0}$ with the common difference  $m$ being sufficiently  divisible, the Frobenius endomorphisms $F^\nu$ act compatibly on certain geometric objects. This ensures that the function $M(\nu)$ has some nice form and converges as $\nu\to\infty$ along $\CP_m$.
This, in turn, forces that $M(\nu)$ is a constant on $\CP_m$ and thereby yields the desired multiplicity formula. 


The theta correspondence and see-saw dual pairs might be used to translate the result on Bessel models \cite{R} to the Fourier-Jacobi case. 
Indeed, such an approach has been employed to study the Gan-Gross-Prasad problems and descent problems over finite fields in \cite{LW1, LW3, LW4, W}. 
Nonetheless, the theta correspondence is an intricate subject (cf. \cite{AMR, LW2, MQZ, P2, P3}). 
Thus, it becomes preferable to provide a more direct and independent proof of the Fourier-Jacobi case.
This is one of the motivations for the current work.

\subsection{The descent of depth zero supercuspidal representations}

The local descent construction (cf. \cite{JS, JNQ, ST})  serves as the inverse of the 
the local Langlands functorial lift from classical groups to general linear groups.
This construction is naturally tied with the non-tempered Gan-Gross-Prasad conjecture (\cite{GGP2}). 
The descent construction from irreducible supercuspidal  self-dual representations
of general linear groups to quasi-split unitary groups in terms of local gamma factors was studied in
\cite{ST}.
Now we describe our second application of \Cref{MF}, which determines the descents of irreducible depth zero supercuspidal representations in terms of supercuspidal types.

Let $\RF$ be a $p$-adic field, and $\RE$ be a quadratic extension of $\RF$.
Let $\CO_E$ denote the ring of integers in $E$, and $\Fp_E$ be the maximal ideal of $\CO_E$. 
Let $\SCEFm$ denote the set of irreducible depth zero supercuspidal representations  of $\GL_m(\RE)$ 
distinguished by $\GL_m(\RF)$. 
As a result of \cite{CG},
$\SCEFm$ is non-empty 
only if $\RE/\RF$ is unramified when $m$ is odd and $\RE/\RF$ is ramified when $m$ is even. In these cases, $\SCEFm$  is parameterized by the orbits of 
\[
\Femoo := \Fe_m^\circ\cap \Fe_m^1
\]
under the $\Gal(\Fe_m/\Fe)$-action. 
In particular, for each $s\in \Femoo$, we attach an irreducible depth zero supercuspidal representation distinguished representation $\tau_s\in \SCEFm$ (see \Cref{thm-dist}). 

Let $g^t$ denote the transpose of a matrix $g$, and let $\iota\in\Gal(\RE/\RF)$ be the nontrivial element that acts entrywise on any matrix over $E$. 
Let $W =\RE^{2n}$ be  a $2n$-dimensional vector space endowed with a skew-Hermitian form
\[
\inn{v_1}{v_2} = v_1^t J_{2n} \iota(v_2) \quad \text{for } v_1,v_2 \in W, 
\]
where 
\[
J_{2n} = \begin{pmatrix} & w_n \\ -w_n & \end{pmatrix}, \quad w_n = \begin{pmatrix} & & & 1 \\ & & 1 & \\
& \iddots & & \\
1 & & & \end{pmatrix}_{n\times n}.
\]
Let $H$ be the quasi-split unitary group $\RU_{2n}(F)$ defined by 
\[
H :=\RU(W) :=  \Set{g\in \GL_{2n}(\RE)  | gJ_{2n}  \iota(g^t) = J_{2n}}.
\]

From now on, we fix a character $\mu$ of $\RE^\times$ such that $\mu|_\RF^\times$ is the quadratic character 
corresponding to the quadratic extension $\RE/\RF$ via the local class field theory.

A lattice $L$ of $W$ is called self-dual if 
\begin{equation}\label{eq:sdlattice}
L =\Set{v'\in W | \langle v', v\rangle \in \CO_\RE \textrm{ for every }v\in L}. 
\end{equation}
The set of self-dual lattices forms a single $H$-orbit. We fix a self-dual lattice $L$ from now on. 
Let 
\[
H_L := \Set{g \in H : g L = L}, \quad
H_{L,+} := \Set{g\in H: g L \subset \Fp_E L} \quad \text{and}\quad 
\sfH_L := H_L/H_{L,+}.
\]
For simplicity, we employ the same symbol to represent an 
 $\sfH_L$-module and its inflation to $H_L$ through the quotient map $H_L\rightarrow \sfH_L$. 

According to the ramification of $\RE/\RF$, we have the following two cases:
\begin{itemize}
\item The extension $\RE/\RF$ is unramified and $m+1 = 2n$: The group $\sfH_L$ is naturally 
isomorphic to the unitary group ${\sf U}_{2n}$ defined over $\Ff$. 
The group $\sfH_L$ has a maximal tori $\sfS_0$
such that the set of $\Ff$-points of the dual tori $\sfS_0^*$ is naturally isomorphic to $\Fe_m^1\times \Fe^1$ up to the $\Gal(\Fe_m/\Ff)\times \Gal(\Fe/\Ff)$-action.   
For $(s,a) \in \Femoo\times \Fe^1$, let $\barsigma_{s,a}$ denote the the cuspidal representation $(-1)^n R^{\sfH_L}_{\sfS_0, (-s,a)}$. 
Then the compactly induced module  
\begin{equation} \label{siga}
\sigma_{s,a} := \cInd^H_{H_L}(\barsigma_{s,a}\otimes \xi_\mu^{-1})
\end{equation}
is a depth zero supercuspidal representation of $H$ (see \cite[Proposition~6.6]{MP}).
Note that the set  
\[
\set{\sigma_{s,a} | a\in \Fe^1}
\]
only depends on the $\Gal(\Fe_m/\Ff)$-orbit of $s$ by \cite[Theorem~6.8]{DL}.  
\item The extension $\RE/\RF$ is ramified and $m=2n$:
The group $\sfH_L$ is naturally isomorphic to the symplectic group $\sfSp_{2n}$ defined over $\Ff$. 
The group
$\sfH_L$ has a maximal tori $\sfS_0$ such that the $\Ff$-points of the dual tori $\sfS_0^*$ is naturally isomorphic to 
$\Fe_{m}^1$ up to $\Gal(\Fe_m/\Ff)$-action. 
For $s \in \Femoo$, let 
$\barsigma_s$ denote the cuspidal representation  $(-1)^n R^{\sfH_L}_{\sfS_0, -s}$. 
Then the compactly induced module 
\begin{equation} \label{sig}
\sigma_s := \cInd^H_{H_L}(\barsigma_s\otimes \xi_\mu^{-1}). 
\end{equation}
is a depth zero supercuspidal representation of $H$. 
Note that the $H$-module $\sigma_s$ only depends on the  $\Gal(\Fe_m/\Ff)$-orbit of $s$. 
\end{itemize}

The local descent construction  $\DD$ is a map sending certain 
irreducible representations of $\GL_m(\RE)$ to representations of $\RU_{2\lfloor\frac{m+1}{2}\rfloor}(\RF)$. 
We refer to \eqref{eq:descent} for the definition of $\DD$ with respect to the choice of $\mu$. 

Now we can state the main result on local descent for unitary groups. 
\begin{thm}\label{thm:DD}
Retain the notation above. Let $s\in \Femoo$, with $m\geq 2$.
\begin{enumerate} 
\item If $\RE/\RF$ is unramified and $m=2n-1$, then  $\DD(\tau_s)$ contains 
$\sigma_{s,a}$ as a multiplicity free direct summand  for each $a\in \Ff_2^1$.
\item If $\RE/\RF$ is ramified and $m=2n$, then
\[
\DD(\tau_s) = \sigma_s. 
\]
\end{enumerate}
\end{thm}

\Cref{thm:DD} can be viewed as an  
explicit realization of supercuspidal representations, which might find
applications in the study of Rankin-Selberg integrals, local L-functions, and
other local factors (see e.g. \cite{AKM+, ST}). 
 
The approach of this paper was first used in \cite{LMNW}, a work in progress, to
study the depth zero local descent for $p$-adic special orthogonal groups in the
sense of \cite{JNQ}, where the upshot was to extend the descent method beyond
the supercuspidal case and thereby give more examples for the non-tempered
Gan-Gross-Prasad conjecture. In future works, we also hope to extend the current
paper and \cite{LMNW} to representations of positive depth.

\delete{
The result in \cite{R} was used in \cite{GR} to verify some cases of the
Gan-Gross-Prasad conjecture for $p$-adic special orthogonal groups. The local
Gan-Gross-Prasad conjecture in the tempered (or generic) case (\cite{GGP1}) has
been proven for $p$-adic local fields through many works which are not recalled
here, but as we mentioned, it remains wide open in the non-tempered case
(\cite{GGP2}).

The descent construction (cf. \cite{JS, JNQ, ST}) seeks the inverse map of the
Langlands functorial lift (from classical groups to general linear groups). This
construction is naturally tied with the non-tempered Gan-Gross-Prasad
conjecture.

In the second part of the paper, we apply \Cref{MF} to deduce explicit descriptions of 
the descent of irreducible supercuspidal depth zero self-dual representations
of general linear groups defined over quasi-split unitary groups in terms of types.

Let $\RE/\RF$ be a quadratic extension of $p$-adic local fields. Let $\tau$ be
an irreducible supercuspidal representation of $\GL_m(\RE)$ such that the local
Asai L-function $L({\sfs}, \tau, {\rm As})$ has a pole at ${\sfs}=1$, or
equivalently that $\tau$ is distinguished by the subgroup $\GL_m(\RF)$. Consider the  
normalized parabolic induction
\[
  \rho_{\tau, 1}:=\Ind^{\RU_{2m}(\RF)}_P(\tau\otimes |\det |_{\RE}^{1/2}),
\]
where $\RU_{2m}(\RF)$ is quasi-split skew-Hermitian and $P$ is the Siegel
parabolic subgroup.
The module $\rho_{\tau,1}$ has a unique irreducible quotient, denoted by
$\pi_\tau$.

The local descent studied in \cite{ST} is concerned with the Fourier-Jacobi
model for the relevant pair of quasi-split skew-Hermitian unitary groups
$\RU_{2m}(\RF)$ and $\RU_{2n}(\RF)$, where $n := \lfloor \frac{m+1}{2}\rfloor$.
Fix a nontrivial additive character $\psi_\RF$ of $\RF$, and a character
$\mu: \RE^\times \to \BC^\times$ such that $\mu\vert_{\RF^\times}$ is 
the quadratic character associated with $\RE/\RF$ via local class field theory.
Associated with these data, the local descent of $\pi_\tau$, denoted by
$\CD_{l_0, \psi_\RF, \mu}(\pi_\tau)$ with $l_0 := m-n$, is a certain twisted
Jacquet module and is a smooth representation of $\RU_{2n}(\RF)$ (see Section
\ref{sec-LDPUG} for the precise definition and the connection with the
non-tempered Gan-Gross-Prasad conjecture).

In \cite{ST} it was proven that $\CD_{l_0, \psi_\RF, \mu}(\tau_\tau)$ is a
multiplicity free direct sum of $\psi_\RF^{-1}$-generic supercuspidal
representations of $\RU_{2n}(\RF)$, and in particular is irreducible if $m=2n$
is even. In this paper, we assume that $\tau$ is of depth zero. Then by the
distinction criterion (\cite{CG, HM1, HM2, S}), it holds that $m=2n-1$ is odd if
$\RE/\RF$ is unramified, and $m=2n$ is even otherwise, and we have the explicit
depth zero supercuspidal data of $\tau$ given by Theorem \ref{thm-dist}.

The second main result of the paper is Theorem \ref{thm:LD}, which describes the
local descent $\CD_{l_0, \psi_\RF, \mu}(\pi_\tau)$ (with $n\geq 2$) explicitly
in terms of the depth zero supercuspidal data. This result is most complete when
$\RE/\RF$ is ramified so that $m=2n$, in which case
$\CD_{l_0, \psi_\RF, \mu}(\pi_\tau)$ is irreducible.
 
To establish Theorem \ref{thm:LD}, we first prove its finite field analog given
by Theorem \ref{U-des} and Theorem \ref{SP-des}, which are two non-regular
examples of Theorem \ref{MF}. Then we analyze the minimal $K$-types of the
induced representation $\rho_{\tau,1}$ and its Langlands quotient $\pi_\tau$, as
given in Section \ref{KTIR}. This enables us to compare the calculation in
\cite{ST} over $p$-adic fields and finite fields in Section \ref{sec-EDZLD}, and
thereby finish the proof. An important ingredient in the proof is the explicit
comparison between the Schr\"odinger model and the generalized lattice model of
the Weil representation of $\RU_{2n}(\RF)$, which was carried out in \cite{P1}
and reformulated in \cite{O}, and we summarize the result in Section
\ref{sec-MSWR}.
 
The explicit construction of supercuspidal representations has found
applications in the study of Rankin-Selberg integrals, local L-functions and
other local factors (see e.g. \cite{AKM+, ST}). Our results should have some
applications along this direction.
 
The approach of this paper was first used in \cite{LMNW}, a work in progress, to
study the depth zero local descent for $p$-adic special orthogonal groups in the
sense of \cite{JNQ}, where the upshot was to extend the descent method beyond
the supercuspidal case and thereby give more examples for the non-tempered
Gan-Gross-Prasad conjecture. In future works we also hope to extend the current
paper and \cite{LMNW} to representations of positive depth.
}

\subsection{Organization of the paper} 
The contents of the paper are divided into two parts: The first, spanning \Cref{sec-P} to \Cref{sec-TNRE}, focuses on the Fourier-Jacobi models for finite classical groups. The second, spanning \Cref{sec-LDPUG} to \Cref{sec-EDZLD},  addresses the depth zero local descent for $p$-adic unitary groups.
 
 \subsection*{Acknowledgement} 
We express our gratitude to Chufeng Nien and Zhicheng Wang for their invaluable insights during the preparation of the work \cite{LMNW}. 
We are also indebted to Dihua Jiang and Lei Zhang for  generously sharing their expertise on descent theory. 
Finally, our sincere thanks go to the anonymous referee whose suggestions significantly improved the presentation of this paper.
 
 D. Liu and F. Shi are supported in part by the National Natural Science Foundation of China (Grant No. 12171421) and Natural Science Foundation of Zhejiang Province (Grant No. LZ22A010006). J.-J. Ma is supported in part by the National Natural Science Foundation of China (Grant No. 11701364 and Grant No. 11971305), Xiamen University Malaysia Research Fund (Grant No. XMUMRF/2022-C9/IMAT/0019), and the Fundamental Research Funds for the Central Universities (No. 20720230022).


\section{Preliminaries} \label{sec-P}

In this section we make some preliminaries from algebraic geometry, and collect some facts about Weil representations of finite classical groups.

\subsection{Functions of geometric type} \label{ss2.1} We introduce  the following notion. 

\begin{defn} \label{geotype}
Let  $\CP\subset \BZ_+$ be an arithmetic progression. A function $M: \CP \to \BC $ is said to be of geometric type if it is of the form
\[
M(\nu) = \frac{\sum^k_{i=1} a_i \alpha_i^\nu}{\sum^l_{j=1} b_j \beta_j^\nu } ,\quad \nu\in\CP,
\]
where $a_i, \alpha_i, b_j, \beta_j \in\BC$, and the denominator is nonzero for every $\nu\in \CP$.  
\end{defn}

We have the following elementary but crucial lemma, for which we will give a proof for completeness.

\begin{lem} \label{const}
Let $M$ be a function of geometric type defined on an arithmetic progression $\CP$. If $M$ is integer-valued and  has a finite limit as $\nu\to\infty$ along $\CP$, then 
$M$ is a constant function. 
\end{lem}

\begin{proof}
By assumption, there is an integer $c$ such that  $M(\nu) = c$ for all $\nu$ sufficiently large. Since $M(\nu)-c$ is also a function of geometric type, it suffices to prove that $M$ is zero when $c=0$.
Write $\CP=\set{\mu, \mu+m, \mu+2m,\ldots}$ with $\mu, m\in \BZ_+$.
Define a function $\tilde{M}$ on the set $\BN$ of natural numbers by $\tilde{M}(d) = M(\mu+md)$ for $d\in \BN$. 
Without loss of generality, we can assume
\[
\tilde{M}(d) = 
\frac{\sum^k_{i=1}a_i' \alpha_i'^d}{\sum^l_{j=1}b_j'\beta_j'^d} \quad \text{for $d\in \BN$}
\]
such that $a_i',\alpha_i',b_j,\beta_j'\in \bC$, and  $\alpha_i'$'s are distinct and non-zero.
The assumption $M(\nu) = 0 $ for $\nu$ sufficiently large implies that $\sum_{i=1}^k a_i' \alpha_i^d = 0$ for $d$ sufficiently large.
Note that the determinant of the $k\times k$ Vandermonde matrix $(\alpha_i'^{j})$ is non-zero. It implies that all $a_i'$'s must be zero, thereby concluding the proof.
\end{proof}

We will focus on functions of geometric type defined on arithmetic progressions $\CP$ starting from $\nu=1$, that is, 
\[
\CP=\CP_m:=\{1+m d : d\geq 0 \}, \quad \text{with } m\in \BZ_+.
\]
If $M$ and $M'$ are two such functions defined on $\CP_m$ and $\CP_{m'}$ respectively, then any linear combination of $M$ and $M'$, and the product $M\cdot M'$, are still functions of geometric type defined on $\CP_m\cap \CP_{m'}=\CP_{\mathrm{lcm}(m, m')}$ where $\mathrm{lcm}(m,m')$ denotes the least common multiple of $m$ and $m'$.

\medskip

The \Cref{geotype} is motivated from algebraic geometry as follows.
Let $\Ff$ be a finite field of characteristic $p$ and cardinality $q$, with a fixed algebraic closure $\bar{\Ff}$. Fix a prime number $\ell$ different from $p$, and fix once for all an identification $\overline{\BQ_\ell}\cong\BC$. For a quasi-projective scheme  $X$  defined over $\Ff$, denote by $ D^b(X, \overline{\BQ_\ell})$ the bounded derived category of constructible $\ell$-adic sheaves on $X$. For $\CF\in D^b(X, \overline{\BQ_\ell})$, the Grothendieck-Lefschetz trace formula (\cite[(2') p471]{SGA5}) relates the local and global traces of the geometric Frobenius $F$: 
\begin{equation} \label{trace}
\sum_{x\in X^{F^\nu}}\sum_i(-1)^i{\rm Tr}(F^\nu_x, H^i(\CF_x)) =\sum_i(-1)^i {\rm Tr}(F^\nu, H^i_c(X\times_\Ff{\bar{\Ff}}, \CF_{\bar\Ff})) \quad \text{ for } \nu\in\BZ_+ 
\end{equation} 
where $\CF_x$ and $\CF_{\bar{\Ff}}$ are the pullbacks of $\CF$ equipped with canonical Frobenius actions $F^\nu_x$ and $F^\nu$ respectively. 

Examining the right-hand side, it is evident that \eqref{trace} is a function of geometric type with respect to $\nu\in\BZ_+$. 
In particular, by setting $\CF$ as the constant sheaf,  one concludes that the cardinality $|X^{F^\nu}|$ is a function of geometric type (defined on  suitable $\CP\subset\BZ_+$).

The inner sum on the left hand side of \eqref{trace} is a function on $X^{F^\nu}$,  
which we denote by 
\begin{equation} \label{shea-fun}
f^{\CF, (\nu)}(x):=\sum_i(-1)^i {\rm Tr}(F^\nu_x, H^i(\CF_x)),\quad x\in X^{F^\nu}.
\end{equation}
The procedure of associating a function to $\CF$ is called 
 is called  Grothendieck's {\it faisceaux-fonction correspondence}.
We favor sheaves over functions because there are more functorial operations on sheaves. 
Notably, the base change operation is crucial in this paper, namely the {\it faisceaux-fonction correspondence}
relates $f^{\CF,(\nu)}(x)$ for $\nu \in \mathbb Z_+$ and their summations \eqref{trace} in a natural way.


\subsection{The Weil representation of a symplectic group and its geometrization}
Assume from now on that $p$ is odd. 
In the following, we view a symplectic 
space of dimension $2n$ defined over $\Ff$ as an additive affine group scheme $V$ defined over $\Ff$ together with a symplectic from
\[
\inn{\ }{\ }_V \colon V \times V \to \BA^1  \quad \text{($\BA^1$ is equipped with the natural additive group structure)}
\]
so that, by restriction to the set $V^F$ of $\Ff$-points,  $\inn{}{}_V$ gives the symplectic structure.

The Heisenberg group associated to $V$ is the scheme 
$\BH_{V} := V\times \BA^1$ equipped with the following group law 
\[
(v,z) (v',z') = (v+v', z+z'+\half \inn{v}{v'}_{V}) \quad \text{for all $v,v'\in V$, and  $z,z'\in \bar{\Ff}$}. 
\]

Now let $\nu \in \BZ_+$. 
Fix a nontrivial additive character $\psi$ of $\Ff$. Then  
\[
\psi^{(\nu)}:=\psi\circ {\rm Tr}_{\Ff_\nu/\Ff}
\]
is a nontrivial additive character of $\Ff_\nu$. 
By \cite[Theorem 2.4 (a)]{G}, up to isomorphism there is a unique irreducible representation $\omega_{\psi}^{(\nu)}$ of $\Sp(V^{F^\nu})\ltimes \BH_{V}^F$, called the Weil representation,  such that
\begin{itemize}
    \item ${\omega_{\psi}^{(\nu)}}\vert_{\BH_{V}^{F^{\nu}}}$ is irreducible and the center of $\BH_{V}^{F^\nu}$ acts by $\psi^{(\nu)}$,
    \item if $\Sp(V^{F^\nu})\cong \SL_2(\BF_3)$, then ${\omega_{\psi}^{(\nu)}}\vert_{\Sp(V^{F^\nu})}$ is non-isomorphic to its complex conjugate.
\end{itemize}

 The Weil representation is geometrized by Gurevich-Hadani \cite{GH}.
 We need the following reformulation of their main result. 
\begin{thm}  \label{geo-weil}
There exists an object  $\CF\in D^b(G\times V, \overline{\mathbb{Q}_\ell})$ of 
pure weight $2n$  such that $\CF[n^2+n]$ is perverse  and  
\[
\omega^{(\nu)}_\psi(g) = f^{\CF, (\nu)}(g,0) , \quad \text{ for $g\in G^{F^\nu}$}.
\]
Here $\omega^{(\nu)}_\psi$ denote the character of the Weil representation and $f^{\CF,(\nu)}$ is the function on $G^{F^\nu}\times V^{F^\nu}$ given by Grothendieck's faisceaux-function correspondence \eqref{shea-fun}. 
\end{thm}
\begin{proof}
By \cite[Theorem 3.2.2.1]{GH} and \cite[(1.2.2)]{GH} 
there is an object $\CK\in D^b(G\times V, \overline{\BQ_\ell})$ of pure weight zero such that $\CK[n^2+n]$ is perverse and  
\[
f^{\CK,(\nu)}(g,0) = \frac{1}{\dim \omega_{\psi}^{(\nu)}} \omega_{\psi}^{(\nu)}(g). 
\]
Since $\dim\omega_{\psi}^{(\nu)} = q^{\nu n}$, we conclude that the $(-n)$-th Tate twist $\CK(-n)$ meets the requirement of the theorem.
See \cite{D} for the theory of weights of $\ell$-adic sheaves.
\end{proof}

A beneficial consequence of \Cref{geo-weil}, which is useful to us later, is that the sums \eqref{Theta} of character values involving $\omega^{(\nu)}_\psi$ are  functions of geometric type by the Grothendieck-Lefschetz trace formula, see \Cref{sec41}.

%

\medskip

The character formula for $\omega_\psi:=\omega_{\psi}^{(1)}$ has been understood through many works (see \cite{Ho2, G, GH, T} for example).
For our purpose, we recall the restriction of $\omega_\psi^{(\nu)}$ to semisimple elements as given in \cite{G}. 


Let us represent a partition 
$\lambda$ in the form $\lambda = (j^{\lambda_j})$, indicating that $\lambda$ contains $\lambda_j$ parts each of size $j$. Then the size of $\lambda$ is given by $|\lambda | = \sum_j j \lambda_j$. 
Following \cite[Section~9]{R}, the $G^F$-conjugacy classes of $F$-stable maximal tori in $G$ are parameterized by pairs of partitions $(\lambda, \lambda')$ such that $|\lambda| + |\lambda'|=n$.
Under this parameterization, an $F$-stable maximal torus $T$ of $G$ corresponds to  a pair of partitions $(\lambda, \lambda')$ such that
\begin{equation} \label{TF}
T^F \cong \prod_j (\Ff_j^\times)^{\lambda_j} \times (\Ff_{2j}^1)^{\lambda'_j}.
\end{equation} 
Let $\vartheta_j$ and $\vartheta_j'$ represent the unique nontrivial quadratic characters of $\Ff_j^\times$ and $\Ff_{2j}^1$ respectively.
Define the quadratic character  $\vartheta_T$ of $T^F$  as the product of the $\vartheta_j$'s and $\vartheta_j'$'s according to the isomorphism \eqref{TF}.
In accordance with \eqref{TF}, for an element $s\in T^F$, the component of $s$ in the $j$-th block is expressed as  
\[
 (s_{j1},\ldots, s_{j\lambda_j}; s'_{j1},\ldots, s'_{j\lambda_j'}).
\]
Following \cite[Corollary 4.8.1]{G}, we have
\begin{equation} \label{char}
\omega_\psi(s) = (-1)^{l(T^F, s)}\vartheta_T(s) q^{\frac{1}{2}\dim V^s},
\end{equation} 
where $V^s =\Ker(s-1_V)$ represents the eigenspace of $s$ acting on $V$ with eigenvalue 1, and
\[
l(T^F, s): =\left| \Set{ (j, k) | 1\leq k \leq \lambda_j', \ s'_{jk}\neq 1} \right|. 
\]

%
%

\subsection{Weil representations of $\RU_n$ and $\GL_n$}
In this section, let $G$ be the unitary group $\RU_n$ or general linear group $\GL_n$ defined over $\Ff$.   
Let $V$ be a symplectic space of dimension $2n$ define over $\Ff$. 
Then $G$ is naturally embedded in $\Sp(V)$ and we let 
$\omega_{\psi}^{(\nu)}$ denote the restriction of the Weil representation of $\Sp(V^{F^{(\nu)}})$ to $G^{F^{\nu}}$.  
As before, let $\omega_{\psi}:= \omega_{\psi}^{(1)}$. 

It is worth noting that the Weil representation of 
of $G^F$ defined in \cite{G}, denoted by $\omega_\psi^\flat$, differs from $\omega_\psi$.
Their relationship can be expressed as  
\[
\omega_\psi^\flat \cong \omega_\psi \otimes \chi_G
,\] 
where $\chi_G$ represents the quadratic character of $G^F$ defined by  
\begin{equation}\label{weil2}
\chi_G(g):=\det(g)^{\frac{q-\varepsilon}{2}}, \quad g\in G^F,
\end{equation}
with $\varepsilon=-1$ for $G=\RU_n$ and $\varepsilon=1$ for $G=\GL_n$.  

Let $T$ be an $F$-stable maximal torus of $G$. Then $T$ is an $F$-stable maximal torus of $\Sp(V)$ via the embedding $G \hookrightarrow \Sp(V)$.
Recall the character $\vartheta_T$ of $T^F$ defined in the previous subsection.
We see that \[
\vartheta_T = \chi_G|_{T^F},
\]
where $\chi_G$ is given in \eqref{weil2}.
We now recall the character values of $\omega_\psi$ when restricted to the tori in $G$. 
We associate the pair $(\lambda,\lambda')$ of partitions such that $|\lambda|+|\lambda'|$ to an $F$-stable maximal torus in $G$ if \eqref{TF} is satisfied. 
This gives an identification of 
the $G^F$-conjugacy classes of $F$-stable maximal tori in $G$ with a subset of the set of pairs of partitions. 
More precisely, we have the following results by \cite[Corollaries 1.4 and 4.8.2]{G}, which can be also deduced from (\ref{char}).
\begin{lem}
\label{AT}
Retain the above notation. 
\begin{enumerate}[label=(\roman*),wide]
\item If $G=\RU_n$, then $\lambda_j=0$ for $j$ odd, and $\lambda_j'=0$ for $j$ even,
and one has
\[
\omega_\psi(s) = (-1)^n \vartheta_T(s) (-q)^{\frac{1}{2}\dim V^s} ,\quad s\in T^F.
\]
\item If $G=\GL_n$, then $\lambda_j'=0$ for all $j$, and one has
\[
\omega_\psi(s) =  \vartheta_T(s) q^{\frac{1}{2}\dim V^s},\quad s\in T^F.
\]
\end{enumerate}
\end{lem}

We conclude this section with the following remarks on the character formula. 

\begin{rmk}  \label{weil-rmk}
Suppose that $G$ is a finite symplectic, unitary, or general linear group. 
\begin{enumerate}[wide]
\item There exists $m\in\BZ_+$ $($see Section 
\ref{ss3.4.1} for the precise condition$)$ such that for $\nu\in\CP_m$, the character formula for $\omega_\psi^{(\nu)}$ at $s\in T^{F^\nu}$ is   similar to \eqref{char} with $q$ replaced by $q^\nu$ and $\vartheta_T$ replaced by 
\[
\vartheta_{T}^{(\nu)}: = \vartheta_T \circ N^T_\nu,
\]
where $N^T_\nu: T^{F^\nu}\to T^F$ is the norm map.

\item For a general element $g\in G^{F^\nu}$, the formula for $\omega_\psi^{(\nu)}(g)$ involves the Weil index of $\psi$, and by \cite[Theorem 4.4]{G} it holds that
\[
\left| \omega^{(\nu)}_\psi(g) \right| = q^{\frac{\nu}{2} \dim V^g},
\]
where $V^g=\Ker(g-1_V)$. 
\end{enumerate}
\end{rmk}

\section{Reeder's multiplicity formula} \label{sec:reeder}

This section is purely expository. We give a brief survey for the general framework provided by Reeder in \cite{R}, which is fundamental  for the study of various branching problems related to Deligne-Lusztig characters of finite groups of Lie type. 

\subsection{Maximal tori}
Let $G$ be a connected reductive $\bar{\Ff}$-algebraic group defined over $\Ff$, with Frobenius $F$. Let $T$ be an $F$-stable maximal torus in $G$, with Weyl group $W_G(T)=N_G(T)/T$. Then $W_G(T)^F=N_G(T)^F/T^F$ by the Lang-Steinberg theorem. Let $s\in G^F$ be a semisimple element, and denote by $G_s := C_G(s)^\circ$ the identity component of the centralizer 
$C_G(s)$ of $s$ in $G$. Put
\[
N_G(s, T)^F :=\Set{\gamma\in G^F | s^\gamma\in T},
\]
where $s^\gamma:=  \gamma^{-1} s \gamma$. Then $G_s^F\times N_G(T)^F$ acts on $N_G(s, T)^F$, and we put
\[
\overline{N}_G(s, T)^F: = G_s^F\backslash N_G(s, T)^F.
\]

A formula for $|\overline{N}_G(s, T)^F|$ is given in \cite{R} as follows. Put $W_G:=W_G(T_0)$ for a fixed $F$-stable maximal torus $T_0$ in $G$ that is contained in an $F$-stable Borel subgroup of $G$. One can associate to $T$ is a cohomology class 
\[
{\rm cl}(T, G) \in H^1(F, W_G).
\]
Define $W_{G_s}$ in a similar way. A map 
\begin{equation} \label{jmap}
j_{G_s}: H^1(F, W_{G_s}) \to H^1(F, W_G)
\end{equation} 
is defined in \cite{R}, which sends the class of an $F$-stable maximal torus in $G_s$  to the class of the same torus in $G$, noting that $G_s$ and $G$ are connected groups of  the same absolute rank.
Denote by $\CT(G)$ the set of all $F$-stable maximal tori in $G$, and for $\omega\in H^1(F, W_G)$ put
\[
\CT_\omega(G) := \Set{T\in \CT(G) | {\rm cl}(T, G)=\omega}.
\]
By \cite[Corollary 2.3]{R}, for $T\in \CT_\omega(G)$ the set $N_G(s, T)^F\neq\varnothing$ if and only if $j_{G_s}^{-1}(\omega)\neq \varnothing$, in which case
\begin{equation} \label{Nbar}
|\overline{N}_G(s, T)^F| = \sum_{ \upsilon \in j_{G_s}^{-1}(\omega) } \frac{|W_G(T)^F|}{|W_{G_s}(T_\upsilon)^F|},
\end{equation} 
where $T_\upsilon$ is an arbitrary member in $\CT_\upsilon(G_s)$, for each $\upsilon \in j_{G_s}^{-1}(\omega)$. 

\subsection{Deligne-Lusztig characters}
An element $g\in G^F$ has the Jordan decomposition $g=su$, where $s\in G^F$ is semisimple and $u\in G^F_s$ is unipotent.  Let $T\in \CT_\omega(G)$ and let $\chi \in \Irr(T^F)$.  The virtual character $R^G_{T,\chi}$ of $G^F$ has the reduction formula (\cite{DL})
\[
R^G_{T,\chi}(su) = \sum_{\bar{\gamma}\in \overline{N}_G(s, T)^F}{}^\gamma \chi(s) Q^{G_s}_{{}^\gamma T}(u), 
\]
where $\gamma\in N_G(s, T)^F$ is a representative of $\bar{\gamma}$, ${}^\gamma T:=\gamma T\gamma^{-1}$,  ${}^{\gamma}\chi:=\chi\circ \Ad(\gamma^{-1})\in \Irr({}^\gamma T^F)$, and $Q^{G_s}_{{}^\gamma T}$ denotes the Green function. Breaking $\overline{N}_G(s, T)^F$ into $W_G(T)^F$-orbits, 
\[
R^G_{T,\chi}(su) = \sum_{\upsilon\in j_{G_s}^{-1}(\omega)} Q^{G_s}_{T_\upsilon}(u) \chi_\upsilon(s),  
\]
where we define $\CO_\upsilon$ to be the $W_G(T)^F$-orbit in $\overline{N}_G(s, T)^F$ corresponding to $\upsilon$ as in \cite[Lemma 2.2]{R}, and 
\begin{equation} \label{chi-O}
\chi_\upsilon:= \sum_{\bar{\gamma}\in \CO_\upsilon}{}^\gamma\chi.
\end{equation}
Note that $\chi_\upsilon$ is a well-defined function on $Z(G_s)^F$, where $Z(G_s)$ denotes the center of $G_s$. It turns out that
\begin{equation} \label{DL}
R^G_{T,\chi}(zu) = \sum_{\upsilon\in j_{G_s}^{-1}(\omega)} Q^{G_s}_{T_\upsilon}(u) \chi_\upsilon(z),\quad \textrm{if }G_z=G_s.  
\end{equation}
For later use, the value $\chi_\upsilon(s)$ can be also unfolded as 
\begin{equation} \label{chi-UF}
\chi_\upsilon(s) = \frac{1}{|W_{G_s}(T_\upsilon)^F|} \sum_{x\in W_G(T)^F} {}^{\gamma x}\chi(s),
\end{equation} 
where $\gamma$ is an arbitrary  element of $N_G(s, T)^F$ such that $\bar{\gamma}\in \CO_\upsilon$. 

\subsection{Multiplicity formula} Let $S\in \CT(G)$. Assume that $f: G^F\to \BC$ is a virtual character,  supported on the set of  elements $g\in G^F$ whose Jordan decomposition $g=su$ satisfies that ${\rm Ad}(G^F)\cdot s \cap S\neq\varnothing$. Let $G^{\rm upt}_s$ be the set of unipotent elements of $G_s$, and let $\CU(G^F_s)$ be the finite set of $\Ad(G_s^F)$-orbits  in $(G^{\rm upt}_s)^F$. By \cite[(5.2)]{R}, 
\begin{equation}  \label{mult1}
\frac{1}{|G^F|}\sum_{g\in G^F} f(g) = \sum_{s\in S^F} \frac{1}{| \overline{N}_G(s, S)^F|} \sum_{[u]\in \CU(G^F_s)} \frac{f(su)}{|C_{G_s}(u)|}.
\end{equation}

Let $I(S)$ be an index set for the  set of subgroups $\Set{G_s | s\in S}$, which is finite. For $\iota\in I(S)$, let $G_\iota$ be the corresponding connected centralizer, and put
\[
S_\iota :=\Set{s\in S | G_s = G_\iota},
\]
so that $S$ has a finite partition 
\[
S = \bigsqcup_{\iota\in I(S)} S_\iota.
\]
The Frobenius $F$ acts on $S$, hence on $I(S)$ as well. Thus \eqref{mult1} has a refinement 
\begin{equation} \label{mult2}
\frac{1}{|G^F|} \sum_{g\in G^F}f(g) = \sum_{\iota\in I(S)^F}  \frac{1}{| \overline{N}_G(\iota, S)^F|}  \sum_{s\in S_\iota^F,  \, [u]\in \CU(G^F_\iota)}   \frac{f(su)}{|C_{G_s}(u)|},
\end{equation}
where $ \overline{N}_G(\iota, S)^F =  \overline{N}_G(s, S)^F$ for any $s\in S_\iota$. 

\subsection{Progression of Frobenius} \label{ss-PF} Recall the arithmetic progressions $\CP_m$, $m\in \BZ_+$, which start from $\nu=1$. We have the following general  remarks from \cite[\S5.4, \S5.5]{R}. 

\subsubsection{} \label{ss3.4.1} Let $\CB_G$ be the flag variety of $G$. For a unipotent class $[u]\in \CU(G^F)$,  let $\CB_G^u$ be the variety of $u$-fixed points. Then $\CB^u_G$ is equi-dimensional, of dimension 
\[
d_G(u): = \dim \CB^u_G = \frac{1}{2} (\dim C_G(u) - \overline{\rm rk} \, G),
\]
where $ \overline{\rm rk} \, G$ denotes the absolute rank of $G$. 

Assume that $p$ is a good prime for $G$. For the classical groups considered in this paper, it amounts to the condition that $p$ is odd. For $T\in \CT(G)$, denote by $Q^G_{T, \nu}$ the Green function for $T$ on $G^{F^\nu}$. Then there exist $m\in \BZ_+$ and Green polynomials $Q_{\omega, u}(t)\in \BZ[t]$ of degree at most $d_G(u)$,  where $\omega\in H^1(F, W_G)$ and $[u]\in \CU(G^F)$, such that the following hold for all $\nu\in \CP_m$:
\begin{itemize}

\item $F^\nu=F$ on $W_G$, and the class ${\rm cl}(T, G)$ is the same with respect to $F$ or $F^\nu$;

\item $F^\nu =F$ on $A_G(C)$, where $C={\rm Ad}(G)\cdot u$ and $A_G(C)$ is the component group of the centralizer of some $F$-fixed element in $C$, and the class of $u$ in $G^F$ or $G^{F^\nu}$ corresponds to the same class in $H^1(F, A_G(C))$; 

\item $Q^G_{T, \nu}(u) = Q_{\omega, u}(q^\nu)$, where $\omega = {\rm cl}(T, G)$.

\end{itemize}
Moreover, if $u=1$ then the leading term of $Q_{\omega, 1}(t)$ is $\epsilon_G(\omega) t^{d_G(1)}$, where 
$\epsilon_G(\omega)=(-1)^{{\rm rk}\, G +{ \rm rk} \, T}$ for $T\in \CT_\omega(G)$,  ${\rm rk}$ denotes the $\Ff$-rank, and $d_G(1) = \dim\CB_G$ is the number of positive roots of $G$.

\subsubsection{}  \label{ss3.4.2} There exist $m\in \BZ_+$, and polynomials $P_{\iota, u}\in\BZ[t]$ of degree equal to $\dim C_{G_\iota}(u)$ and leading coefficient 
$|A_\iota(u)|$, where $\iota\in I(S)^F$, $[u]\in \CU(G^F_\iota)$ and $A_\iota(u)$ is the component group of $C_{G_\iota}(u)$,  such that the following hold for all $\nu \in \CP_m$: 
\begin{itemize}

\item $F^\nu = F$ on $I(S)$, and the conditions on $m$ in Section \ref{ss3.4.1} hold for every $G_\iota$, $\iota\in I(S)$; 

\item
$| C_{G_\iota}(u)^{F^\nu}| = P_{\iota, u}(q^\nu)$;

\item
$
|\overline{N}_G(\iota, S)^{F^\nu}| = |\overline{N}_G(\iota, S)^F|.
$
\end{itemize}

\section{Multiplicity formula for Fourier-Jacobi models}

Assume from now on that $G$ is one of the classical groups $\Sp_{2n}$, $\RU_n$ or $\GL_n$ as before. In this section, following Reeder's method \cite{R} we derive a formula for  the multiplicity 
\[
M(1): =\langle R^G_{T, \chi} \otimes\omega_\psi^\vee, R^G_{S, \eta}\rangle_{G^F}
\]
associated to two Deligne-Lusztig characters, where $S, T\in \CT(G)$, $\chi\in \Irr(T^F)$ and $\eta\in \Irr(S^F)$. We outline the strategy as follows.

The method is to show that the integer valued function 
\begin{equation} \label{M}
M(\nu): = \langle R^G_{T, \chi^{(\nu)}} \otimes\omega_\psi^{(\nu), \vee}, R^G_{S, \eta^{(\nu)}}\rangle_{G^{F^\nu}}
\end{equation}
is of geometric type and has a finite limit as $\nu\to\infty$ along some arithmetic progression $\CP_m$ starting from $\nu=1$. Here 
$\chi^{(\nu)}:=\chi\circ N^T_\nu$ is a character of $T^{F^\nu}$, and likewise $\eta^{(\nu)}$ is a character of $S^{F^\nu}$.  This implies that $M$ is constant on $\CP_m$ by Lemma \ref{const}, and therefore 
\[
M(1) = \lim_{\nu\to\infty, \,  \nu\in\CP_m}M(\nu).
\]
The main result is stated as Theorem \ref{MF}, and an explicit formula for regular Deligne-Lusztig characters is given by \eqref{reg-MF}. As an application we prove a conjecture 
of Hiss and Schr\"oer \cite{HS} in type A, which is presented as Theorem \ref{HS-A}.

For convenience, we write $\nu\xrightarrow{\CP_m}\infty$ to indicate that $\nu\to\infty$ along $\CP_m$.  For two functions $A(\nu)$ and $B(\nu)$ defined on $\CP_m$, we write 
\[
A(\nu) \approx_{\nu, \CP_m} B(\nu)
\]
if $A(\nu) = B(\nu) C(\nu)$ for some function $C(\nu)$ which converges to 1 as $\nu\xrightarrow{\CP_m}\infty$.

\subsection{Multiplicity as a function of geometric type}  We first prove the following result. 
\label{sec41}
\begin{prop}\label{prop41}
Let $m\in\BZ_+$ be as in Section \ref{ss3.4.2}. Then  \eqref{M} defines  a function $M$ of geometric type on $\CP_m$.
\end{prop}

\begin{proof}
Define the function 
\[
f^{(\nu)}: G^{F^\nu} \to \BC,\quad
f(g) = \overline{R^G_{T, \chi^{(\nu)}}(g)}\cdot R^G_{S,\eta^{(\nu)}}(g)\cdot \omega_\psi^{(\nu)}(g).
\]
Then we have that
\[
M(\nu) = \frac{1}{|G^{F^\nu}|} \sum_{g\in G^{F^\nu}} f^{(\nu)}(g). 
\]
By \eqref{DL} and \eqref{mult2}, we find that
\[
M(1) =  \sum_{\iota\in I(S)^F, \, [u]\in \CU(G^F_\iota)} \sum_{\upsilon, \, \varsigma} \frac{\overline{Q^{G_\iota}_{T_\upsilon}(u)} Q^{G_\iota}_{S_\varsigma}(u)}{| \overline{N}_G(\iota, S)^F| \,
| C_{G_\iota}(u)^F| } \sum_{s\in S^F_\iota} \overline{\chi_{\upsilon}(s)} \cdot \eta_{\varsigma}(s) \cdot \omega_{\psi}(su),
\]
where the middle sum is over $\upsilon \in j^{-1}_{G_\iota}({\rm cl}(T, G))$ and $\varsigma\in j^{-1}_{G_\iota}({\rm cl}(S, G))$. 

Denote $\alpha = (\iota, u, \upsilon, \varsigma)$ the summation indices of quadruples where
\begin{equation} \label{alpha}
\iota\in I(S)^F, \quad [u] \in \CU(G^F_\iota),\quad \upsilon \in j^{-1}_{G_\iota}({\rm cl}(T, G)), \quad \varsigma\in j^{-1}_{G_\iota}({\rm cl}(S, G)).
\end{equation} 
They are unchanged if $F$ is replaced by $F^\nu$ with $\nu\in \CP_m$, thus for such $\nu$ we have that
\[
M(\nu) = \sum_\alpha \Psi_\alpha(q^\nu) \Theta_\alpha(\nu),
\]
where
\begin{equation} \label{Psi}
\Psi_\alpha(t)  := \frac{\overline{Q^{G_\iota}_{\upsilon, u}(t)} Q^{G_\iota}_{\varsigma, u}(t)}{|\overline{N}_G(\iota, S)^F| \, | P_{\iota, u}(t)|}
\end{equation} 
is a rational function of $t$, and 
\begin{equation} \label{Theta}
\Theta_\alpha(\nu) :=  \sum_{s\in S^{F^\nu}_\iota} \overline{\chi_{\upsilon}^{(\nu)}(s)}\cdot \eta_{\varsigma}^{(\nu)}(s)\cdot\omega_{\psi}^{(\nu)}(su). 
\end{equation}
Here $Q^{G_\iota}_{\upsilon, u}(t)$ and $Q^{G_\iota}_{\varsigma, u}(t)$ are the Green polynomials from Section \ref{ss3.4.1}, and $P_{\iota, u}(t)$ is the polynomial from Section \ref{ss3.4.2}.
Note that 
\[
S_\iota\subset Z_\iota:= Z(G_\iota).
\]
Recall from \eqref{chi-O} that  $\chi_{\upsilon}$ and $\eta_{\varsigma}$ are sums of characters on $Z_\iota$,  thus the values $\chi_{\upsilon}^{(\nu)}(s)$ and  $\eta_{\varsigma}^{(\nu)}(s)$, $s\in S_\iota^{F^\nu}$, are the local traces of $F^\nu_s$ on the corresponding sheaves. By the geometrization Theorem \ref{geo-weil} and the Grothendieck-Lefschetz trace formula,  each $\Theta_\alpha(\nu)$ is a function of geometric type on $\BZ_+$. More precisely,
 denote by $\CL_{{}^\gamma\chi}$ and $\CL_{{}^{\delta}\eta}$ the rank one sheaves 
on $Z_\iota$ corresponding to the characters ${}^\gamma\chi$ and ${}^{\delta}\eta$ respectively,
where $\overline{\gamma}\in \CO_\upsilon$ and $\overline{\delta}\in \CO_{\varsigma}$. Then 
\[
\sum_{s\in S^{F^\nu}_\iota} \overline{{}^\gamma\chi^{(\nu)}(s)}\cdot {}^{\delta}\eta^{(\nu)}(s)\cdot \omega_{\psi}^{(\nu)}(su) ={\rm Tr}\left(F^\nu, R\Gamma_c(S_\iota,
u^*\CF\otimes \CL_{{}^\gamma\chi}^\vee\otimes \CL_{{}^{\delta}\eta})\right),
\]
where $\CF$ is the sheave given in Theorem \ref{geo-weil}, and $u: G\to G, x\mapsto xu$ is the right translation by $u$.  
It follows that $M(\nu)$ is a function of geometric type on $\CP_m$. 
\end{proof}

Thus it remains to show that, after further increasing the divisibility of $m$ if necessary, $M(\nu)$ converges as $\nu\xrightarrow{\CP_m}\infty$, and evaluate its limit. 

\subsection{A partition of $S$} \label{PS} For the classical groups $G$ of our concern, the centralizers $C_G(s)$, $s\in S$, are connected hence $G_s= C_G(s)$. We first give the description of $G_s$ uniformly, following e.g. \cite{AMR}. 

Recall that $G$ acts on the $2n$-dimensional symplectic space $V$. Define a group 
\[
\Gamma: = \Gal(\bar{\Ff}/\Ff)\times \langle \imath \rangle,
\]
where $\imath$ is the involution  $\imath: \bar{\Ff}^\times \to \bar{\Ff}^\times$, $x\mapsto x^{-1}$. Then $\Gamma$ naturally acts on $\bar{\Ff}^\times$. Denote by $\Lambda_s \subset \bar{\Ff}^\times$ the set of distinguished eigenvalues of $s$ acting on $V$.
Clearly $\Gamma$ acts on $\Lambda_s$, and for $a\in \Lambda_s$ denote by $[a]$ the $\Gamma$-orbit of $a$. 
Then we have that
\[
G_s = \prod_{[a]\in \Lambda_s/\Gamma} G_{s, [a]},
\]
where 
\begin{itemize}

\item $G_{s, [1]}$ and $G_{s, [-1]}$ are classical groups of the same type as $G$;

\item $G_{s, [a]}$, $[a]\neq[\pm 1]$, are general linear groups or unitary groups.

\end{itemize}
In the above, $G_{s, [\pm1]}$ is interpreted as the trivial group if $\pm 1\not\in \Lambda_s$. Let us write 
\begin{equation} \label{G'}
G_s' : = \prod_{[a]\neq [ 1]} G_{s,[a]},
\end{equation} 
so that $G_s = G_s' \times G_{s, [1]}$.

Let $J(S)$ be an index set for the following set of symplectic subspaces of $V$:
\[
\Set{V^s | s\in S}.
\]
Then $J(S)$ is finite of cardinality $|J(S)|= 2^n$. The Frobenius $F$ acts naturally on $J(S)$, and it can be shown that the set $J(S)^F$ of fixed points corresponds to
$$ \Set{ V^s | s \in S^F}.
$$
 Following \eqref{TF} and \Cref{AT}, if we assume that $S^F$ is of the form
\begin{equation} \label{SF}
S^F\cong \prod_j (\Ff_j^\times)^{\mu_j}\times (\Ff^1_{2j})^{\mu_j'},\quad |\mu|+|\mu'|=n,
\end{equation}
then $J(S)^F$ can be taken to be the set of all subsets of $\prod_j [1, \mu_j]\times [1,\mu_j']$ hence $|J(S)^F|= 2^{\sum_j (\mu_j+\mu_j')}$. 
For $\jmath\in J(S)$, denote by $V^\jmath$ the corresponding symplectic subspace of $V$, and put
\[
S_\jmath := \Set{s\in S | V^s= V^\jmath}.
\]
Then we have a  partition 
\[
S= \bigsqcup_{\jmath\in J(S)} S_\jmath. 
\]
We now increase the divisibility of $m$ such that $F^m$ acts trivially on $J(S)$. Then for any
$\nu\in \CP_m$ it holds that
\[
S^{F^\nu}= \bigsqcup_{\jmath\in J(S)^F} S_\jmath^{F^\nu}.
\]
Denote by $G_\jmath$ the centralizer of $S_\jmath$ in $G$:
\begin{equation} \label{ZJ}
G_\jmath:= C_G(S_\jmath) = Z_\jmath \times G_{\jmath, [1]},
\end{equation}
where $Z_\jmath :=\overline{S_\jmath}\subset S$ can be viewed as  a maximal torus in $\Sp(V/V^\jmath)$, and $G_{\jmath, [1]} = G_{s, [1]}\subset \Sp(V^\jmath)$ for any $s\in S_\jmath$. 

Recall the index set $I(S)$.  For each $\iota\in I(S)^F$, let $J(S_\iota)\subset J(S)$ be the index set for the following set of symplectic subspaces of $V$:
\[
\Set{V^s | s\in S_\iota}.
\]
Note that by the description of $G_s$,  we have that $|J(S_\iota)^F|\leq 2$ if $G=\Sp_{2n}$, and $|J(S_\iota)^F|\leq |\Lambda_s/\Gamma|$, $s\in S_\iota$, if $G=\RU_n$ or $\GL_n$. For $\jmath\in J(S_\iota)$,  put
\[
S_{\iota, \jmath}:=S_\iota\cap S_\jmath = \Set{s\in S_\iota | V^s=V^\jmath}.
\]
This gives a finite partition 
\begin{equation} \label{S-part}
S_\iota = \bigsqcup_{\jmath\in J(S_\iota)}S_{\iota, \jmath},
\end{equation} 
which induces a partition for each $\nu\in \CP_m$:
$$
S_\iota^{F^\nu} = \bigsqcup_{\jmath\in J(S_\iota)^F}S_{ \iota, \jmath}^{F^\nu}.
$$

For each $\jmath\in J(S_\iota)^F$ we have a decomposition 
\[
G_\iota=G'_{\iota, \jmath} \times G_{ \jmath, [1]},
\]
where $G'_{\iota, \jmath} =G'_s$  for any $s\in S_{ \iota, \jmath}$. Note that $G'_\jmath$  can be naturally viewed as a subgroup 
of  $\Sp(V/V^\jmath)$, and we have the inclusion 
\begin{equation} \label{incl}
S_{ \iota, \jmath} \subset Z_{\iota, \jmath}: = Z(G_{\iota, \jmath}') \subset Z_{\jmath}\subset G_{\iota, \jmath}'.
\end{equation}

\subsection{Multiplicity formula} 

\subsubsection{} From Section \ref{ss-PF}, it follows that 
\[
\deg \Psi_\alpha(t) \leq - \overline{\rm rk} \, G_\iota =  - \overline{\rm rk}\, G= -n.
\]
By \eqref{S-part}, we can write the function \eqref{Theta} as
\[
\Theta_\alpha(\nu ) = \sum_{\jmath\in J(S_\iota)} \Theta_{\alpha, \jmath} (\nu),
\]
where
\[
\Theta_{\alpha, \jmath}(\nu): = \sum_{s\in S^{F^\nu}_{\iota, \jmath}} \overline{\chi_{\upsilon}^{(\nu)}(s)}\cdot \eta_{\varsigma}^{(\nu)}(s)\cdot \omega_{\psi}^{(\nu)}(su).
\]
By Remark \ref{weil-rmk} (2), at $s\in S^{F^\nu}_\jmath$ we have 
\begin{equation} \label{est-weil}
\left| \omega_{\psi}^{(\nu)}(su)\right| = q^{\frac{\nu}{2}\dim V^{su}}\leq  q^{\frac{\nu}{2}\dim V^{s}} = q^{\nu\cdot \overline{\rm rk}\, G_{\jmath, [1]}}.
\end{equation}
This gives the estimate 
\begin{equation} \label{est-theta}
\left| \Theta_{\alpha, \jmath}(\nu)\right| \leq \left| Z_{ \iota, \jmath}^{F^\nu} \right| \cdot q^{\nu \cdot \overline{\rm rk}\, G_{\jmath, [1]}} \approx_{\nu, \CP_m} q^{\nu(\overline{\rm rk}\, Z_{\iota, \jmath} + \overline{\rm rk}\, G_{\jmath, [1]})} \leq q^{\nu \cdot \overline{\rm rk}\, G_\iota} = q^{\nu n}.
\end{equation} 
It follows that
\[
\Psi_\alpha(\nu) \Theta_{\alpha, \jmath}(\nu)\to 0,\quad {\rm as } \ \nu\xrightarrow{\CP_m}\infty
\]
unless 
\[
\overline{\rm rk}\, Z_{\iota, \jmath} = \overline{\rm rk}\, G_{\iota, \jmath}',
\] 
which is equivalent to that $G_{\iota, \jmath}' = Z_\jmath$ in view of \eqref{incl}.
At this point, we make the following 

\begin{rmk} \label{rmk-reg}

\begin{enumerate}
\item For $\iota\in I(S)^F$ and $\jmath\in J(S_\iota)^F$, the following are equivalent:

\begin{itemize}
\item $G_{\iota, \jmath}' = Z_\jmath$,
\item $G_\iota = G_\jmath$,
\item  $S_{\iota, \jmath}= Z_{\jmath}^\reg$,
\end{itemize}
where $Z_{\jmath}^\reg$ is  the subset of elements of $Z_\jmath$ that are regular in   ${\rm Sp}(V/V^\jmath)$.
\item For every $\jmath\in J(S)^F$, there is a unique $\iota\in I(S)^F$ such that $\jmath\in J(S_\iota)^F$ and the equivalent conditions in {\rm (1)} hold. This gives a map
\[
\phi: J(S)^F\to I(S)^F
\]
such that $G_\jmath = G_{\phi(\jmath)}$, $\jmath\in J(S)^F$.
\end{enumerate}
\end{rmk}

\subsubsection{} Assume that  $G_\iota=G_\jmath$, i.e. $\iota= \phi(\jmath)$ where
$\phi$ is the map in Remark \ref{rmk-reg} (2). Then  $[u]\in \CU(G^F_\iota) = \CU(G^F_{\jmath,[1]})$. In this case, if $u\neq 1$ then for  $s\in S_{\iota, \jmath} = Z_\jmath^\reg$ we have that
\[
\dim V^{su}=\dim (V^s\cap V^u) < \dim V^s.
\]
It follows from \eqref{est-weil} and a similar estimate as \eqref{est-theta} that
\[
\Psi_\alpha(\nu) \Theta_{\alpha, \jmath}(\nu)\to 0,\quad {\rm as } \ \nu\xrightarrow{\CP_m}\infty.
\]

Thus we further assume that $u=1$. By the character formula \eqref{char}, at $s\in Z_\jmath^{\reg, F^\nu}$  we have 
\[
\omega_\psi^{(\nu)}(s) = (-1)^{l(Z_\jmath^{F^\nu}, s)} \vartheta_{Z_\jmath}^{(\nu)}(s) q^{\frac{\nu}{2} \dim V^s} .
\]
The properties of $m$ listed in Section \ref{ss3.4.2} ensure that $l(Z_\jmath^{F^\nu}, s) = l(Z^F_\jmath, s)$ for all $\nu\in \CP_m$, which is the number of anisotropic factors of $Z_\jmath$. Denote this number by $l(Z_\jmath)$. By Remark \ref{rmk-reg} and the fact that regular elements are Zariski open dense, we find that
\begin{align*}
\Theta_{\alpha, \jmath}(\nu) &  \approx_{\nu, \CP_m} (-1)^{l(Z_\jmath)}  q^{\nu\cdot \overline{\rm rk}\, G_{\jmath, [1]}} \left| Z_\jmath^{F^\nu}\right| \langle \chi_\upsilon^{(\nu)}\vartheta_{Z_\jmath}^{(\nu)}, \eta_\varsigma^{(\nu)}\rangle_{Z_\jmath^{F^\nu}} \\
& \approx_{\nu, \CP_m} (-1)^{l(Z_\jmath)}  \langle \chi_\upsilon \vartheta_{Z_\jmath}, \eta_\varsigma\rangle_{Z_\jmath^F} q^{\nu n} .
\end{align*}
On the other hand, by Section \ref{ss-PF} when $u=1$ we have $\deg\Psi_\alpha(t)= -n$ and 
\[
\Psi_\alpha(q^\nu) \approx_{\nu, \CP_m} \frac{(-1)^{{\rm rk}\, T +{\rm rk}\, S}}{ \left| \overline{N}_G(\iota, S)^F \right| } q^{-\nu n},
\]
noting that $A_\iota(1)$ is trivial. Thus
\begin{equation} \label{summand}
\Psi_\alpha(q^\nu) \Theta_{\alpha, \jmath}(\nu)  \to  \frac{(-1)^{{\rm rk}\, T +{\rm rk}\, S + l(Z_\jmath)}}{ \left| \overline{N}_G(\iota, S)^F \right| }
\langle \chi_\upsilon \vartheta_{Z_\jmath}, \eta_\varsigma\rangle_{Z_\jmath^F}\quad {\rm as} \ 
\nu\xrightarrow{\CP_m}\infty. 
\end{equation} 
Since in this case $G_\iota = G_\jmath = Z_\jmath \times G_{\jmath, [1]}$, we observe that 
\[
\left| j^{-1}_{G_\jmath }({\rm cl}(S, G)) \right| =1 \quad{\rm and}\quad  \left| j^{-1}_{G_\jmath }({\rm cl}(T, G)) \right|   \leq 1,
\]
with equality hold in the latter if and only if $G_\jmath$ contains a $G^F$-conjugate of $T$. In this case we denote the unique elements of $j^{-1}_{G_\jmath }({\rm cl}(T, G))$ and $j^{-1}_{G_\jmath }({\rm cl}(S, G))$ by $\upsilon(\jmath)$ and $\varsigma(\jmath)$ respectively. 

\subsubsection{} Now we can present the multiplicity formula.

\begin{thm} \label{MF} Let the definitions be as above. Then we have the multiplicity 
\begin{align*}
\langle R^G_{T, \chi} \otimes\omega_\psi^\vee, R^G_{S, \eta}\rangle_{G^F} & = \sum_{\jmath\in J(S, T)}\frac{(-1)^{{\rm rk}\, T +{\rm rk} \, S+ l(Z_\jmath)}}{|\overline{N}(\phi(\jmath), S)^F|}
\langle \chi_{\upsilon(\jmath)}\vartheta_{Z_\jmath}, \eta_{\varsigma(\jmath)}\rangle_{Z_\jmath^F} \\
& = \sum_{\jmath\in J(S, T)}\frac{(-1)^{{\rm rk}\, T +{\rm rk} \, S+ l(Z_\jmath)}}{|W_{G_\jmath}(T_\jmath)^F| \, |W_G(S)^F|}\sum_{w\in W_G(T_\jmath)^F, \, v\in W_G(S)^F} \langle {}^w\chi_\jmath \vartheta_{Z_\jmath}, {}^v\eta\rangle_{Z_\jmath^F},
\end{align*} 
where $J(S, T):=\Set{\jmath\in J(S)^F | j_{G_\jmath}^{-1}({\rm cl}(T, G)) \neq \varnothing }$, $j_{G_\jmath}^{-1}({\rm cl}(T, G)) =\{\upsilon(\jmath)\}$, $j_{G_\jmath}^{-1}({\rm cl}(S, G)) = \{\varsigma(\jmath)\}$, and $(T_\jmath, \chi_\jmath)$ is any $G^F$-conjugate of $(T, \chi)$ such that 
$T_\jmath\subset G_\jmath$. 
\end{thm}

\begin{proof} In view of the above discussions, we see that each summand of 
\[
M(\nu) =\sum_{\alpha} \sum_{\jmath\in J(S_\iota)^F} \Psi_\alpha(q^\nu)\Theta_{\alpha, \jmath}(\nu)
\]
converges as $\nu\xrightarrow{\CP_m}\infty$. 
For each $\jmath \in J(S)^F$, there is at most one quadruple 
\[
\alpha = (\phi(\jmath), [1], \upsilon(\jmath), \varsigma(\jmath)),
\]
which exists only if $ j^{-1}_{G_\jmath }({\rm cl}(T, G))\neq\varnothing$, such that $ \Psi_\alpha(q^\nu)\Theta_{\alpha, \jmath}(\nu)$ contributes to the limit of $M$ and is given by
\eqref{summand}. Thus the double sum $M(\nu)$, after taking limit as $\nu\xrightarrow{\CP_m}\infty$, reduces to a sum over $\jmath\in J(S, T)$. The  formula in the theorem follows from direct substitutions of   \eqref{Nbar} and \eqref{chi-UF}.
\end{proof}

\subsection{Regular case} \label{ssec-RC}

Let $T, S\in \CT(G)$ be parametrized by pairs of partitions $(\lambda, \lambda')$ and $(\mu, \mu')$ as \eqref{TF} and \eqref{SF} respectively,
subject to the constraints in \Cref{AT}. 
Likewise, for $\jmath\in J(S, T)$, $Z_\jmath^F$ is precisely of the form  
\[
Z_\jmath^F \cong \prod_j (\Ff^\times_j)^{\nu_j} \times (\Ff^1_{2j})^{\nu_j'},
\]
where $\nu_j\leq \lambda_j, \mu_j$ and $\nu_j'\leq \lambda_j', \mu_j'$ for each $j$.  Then
\begin{equation} \label{lZ} 
l(Z_\jmath) = \sum_j \nu_j'.
\end{equation} 
There are 
\begin{equation} \label{JST}
{\mu \choose \nu}{\mu' \choose \nu'}: = \prod_j {\mu_j\choose \nu_j} {\mu_j'\choose \nu_j'}
\end{equation} 
elements $\jmath\in J(S, T)$ that give rise to $(\nu, \nu')$. Here the notation for the partition $\nu=(j^{\nu_j})$ should not be confused with the variable $\nu\in\CP_m$ used earlier. Put
\[
\kappa_G := \begin{cases} 1, & \textrm{if }G\textrm{ is of type A}, \\
2, & \textrm{if }G\textrm{ is of type C}. 
\end{cases}
\]
We have that
\begin{equation} \label{NJS}
|\overline{N}(\phi(\jmath), S)^F| = \frac{|W_G(S)^F|}{|W_{G_\jmath}(S)^F|} ={\mu\choose \nu} {\mu'\choose \nu'}\prod_j \nu_j!\nu_j'! (\kappa_G \cdot j)^{\nu_j + \nu_j'}.
\end{equation}

Assume now that $\chi$ and $\eta$ are {\it regular}, that is, they have trivial stabilizers in $W_G(T)^F$ and $W_G(S)^F$ respectively. In this case $(-1)^{{\rm rk}\, T + {\rm rk}\, G}R^G_{T,\chi}$ and $(-1)^{{\rm rk}\, S + {\rm rk}\, G}R^G_{S,\eta}$ are irreducible representations of $G^F$. Below we give an explicit formula for the multiplicity 
\[
\langle R^G_{T, \chi} \otimes\omega_\psi^\vee, R^G_{S, \eta}\rangle_{G^F}
\]
using Theorem \ref{MF}. As expected, the computation is similar to the case of Bessel models for special orthogonal groups given in \cite[\S9]{R}, and we  give a sketch. 

Recall the group $\Gamma=\Gal(\bar{\Ff}/\Ff)\times\langle\imath\rangle$ from Section \ref{PS}. It acts on each $\Irr(\Ff_j^\times)$ and $\Irr(\Ff_{2j}^1)$ in the obvious way.  Put
\[
\Gamma_G :=\begin{cases} \Gal(\bar{\Ff}/\Ff), & \textrm{if }G\textrm{ is of type A}, \\
\Gamma, & \textrm{if }G\textrm{ is of type C}.\end{cases}
\]
On the $j$-th blocks of $T^F$ and $S^F$ respectively, write
\begin{align*}
& \chi=\chi_{j1}\otimes\cdots \otimes \chi_{j\lambda_j}\otimes \chi_{j1}'\otimes\cdots\otimes \chi_{j\lambda_j'}', \\
&  \eta=\eta_{j1}\otimes\cdots \otimes \eta_{j\mu_j}\otimes \eta_{j1}'\otimes\cdots\otimes \eta_{j\mu_j'}'. 
\end{align*}
Recall the quadratic characters $\vartheta_j$ and $\vartheta_j'$ of $\Ff_j^\times$ and $\Ff_{2j}^1$ respectively. Define
\begin{align*}
& I_j :=\Set{ k\in [1, \mu_j] | \eta_{jk}\in \Gamma_G\cdot \chi_{jl}\vartheta_j \textrm{ for some }l \in [1, \lambda_j]}, \\
& I_j' :=\Set{ k\in [1, \mu'_j] | \eta'_{jk}\in \Gamma_G\cdot \chi'_{jl}\vartheta_j' \textrm{ for some }l \in [1, \lambda'_j] }.
\end{align*}
By the regularity assumption, 
\begin{equation} \label{CS}
\langle \chi_{\upsilon(\jmath)}\vartheta_{Z_\jmath}, \eta_{\varsigma(\jmath)}\rangle_{Z_\jmath^F}= \prod_j {|I_j| \choose \nu_j}{|I_j'| \choose \nu_j'} \nu_j! \nu_j'!(\kappa_G \cdot j)^{\nu_j+ \nu_j'}.
\end{equation} 
Put  \label{eTS}
\begin{equation} 
e_{T, S} :=(-1)^{{\rm rk}\, T +{\rm rk}\, S}.
\end{equation}
 By  \eqref{lZ}, \eqref{JST}, \eqref{NJS}, \eqref{CS} and Theorem \ref{MF},
\begin{equation} \label{reg-MF}
e_{T, S}\cdot\langle R^G_{T,\chi}\otimes \omega^\vee_\psi, R^G_{S,\eta}\rangle_{G^F} =  \sum_{\nu, \nu'} (-1)^{\sum_j \nu_j'} \prod_j {|I_j| \choose \nu_j } {|I_j'| \choose \nu_j'} = \begin{cases} 2^r, & \textrm{if }I_j'=\varnothing\textrm{ for all }j, \\
0, & \textrm{otherwise,}
\end{cases}
\end{equation} 
where $r= \sum_j |I_j|$. 

In particular, further assume that $T$ and $S$ are anisotropic mod $Z(G)$, so that $(-1)^{{\rm rk}\, T + {\rm rk}\, G}R^G_{T,\chi}$ and $(-1)^{{\rm rk}\, S + {\rm rk}\, G}R^G_{S,\eta}$ are cuspidal. Then $r=0$ when $G= \Sp_{2n}$ or $\RU_n$, and $r\leq 1$ when $G=\GL_n$. Note that $I_j'=\varnothing$ for all $j$ in the latter case. 

\subsection{A conjecture of Hiss and Schr\"oer}

The study of the multiplicity \eqref{FJ-mult}  of the basic  Fourier-Jacobi model amounts to understanding the irreducible decomposition 
\[
\pi\otimes\omega_\psi^\vee = \sum_{\sigma\in \Irr(G^F)} m(\pi, \sigma) \sigma
\]
for any $\pi\in \Irr(G^F)$. The case that $\pi ={\rm St}_G$ is the Steinberg representation of $G^F$ is studied in \cite{HZ}. Two general conjectures are made in \cite{HS}: one is about the structure of the algebra $\End_{\BC G^F}(\pi\otimes\omega_\psi^\vee)$, and the other one is the following conjecture on the multiplicities $m(\pi, \sigma)$ in the decomposition of $\pi\otimes\omega_\psi^\vee$. We note that $\omega_\psi^\vee\cong \omega_{\psi^{-1}}$. 

\begin{conj}[\cite{HS}] \label{HS-conj}
Let $\pi\in\Irr(G^F)$, where $G=\Sp_{2n}$, $\RU_n$ or $\GL_n$ is defined over $\Ff$. Then the multiplicities of the irreducible constituents of $\pi\otimes \omega_\psi^\vee$ are bounded by a function of $n$ which is independent of $q=|\Ff|$.  
\end{conj}

Conjecture \ref{HS-conj} is stated only for $G=\Sp_{2n}$ or $\RU_n$ in \cite{HS}, but it is natural to include the $\GL_n$ case as well. Below we apply Theorem \ref{MF} and the classification of $\Irr(G^F)$ in \cite{LS} to give a quick proof of this conjecture  for $G=\RU_n$ or $\GL_n$. We restrict to type A because in this case virtual characters  on $G^F$ are all {\it uniform}, i.e. they can be spanned by Deligne-Lusztig characters. It is also possible to prove the conjecture for $G=\Sp_{2n}$ using Theorem \ref{MF}, the classification in \cite{L1} and some extra work (e.g. \cite{W}), which will not be pursued here.

We now  prove  Conjecture \ref{HS-conj} in type A with the following explicit bound for the multiplicities $m(\pi,\sigma)$, although  it is by no means optimal. 

\begin{thm} \label{HS-A}
Assume that $G=\RU_n$ or $\GL_n$ is defined over $\Ff$.  Then for any $\pi, \sigma\in \Irr(G^F)$, 
\[
m(\pi, \sigma) = \langle \pi\otimes \omega_\psi^\vee, \sigma\rangle_{G^F} \leq 2^n(n!)^2 .
\]
\end{thm}

\begin{proof}
By \cite[Theorem 3.2]{LS}, any irreducible representation of $G^F$ is of the form 
\[
R^G_{ L, \chi,  \rho }:=\frac{(-1)^{{\rm rk}\, G + {\rm rk}\, L}}{|W_L|}\sum_{w\in W_L} \textrm{Tr}(w\tilde{F}, V_\rho) \cdot R^G_{T_w, \chi_w},
\]
where 
\begin{itemize}
\item $L$ is an $F$-stable  reductive connected  subgroup of $G$ of absolute rank equal to that of $G$,
\item $\chi: L^F\to \BC^\times$ is a homomorphism, satisfying certain properties,
\item $(\rho, V_\rho) \in \Irr(W_L)^F$ and $\tilde{F}$ is a  finite order automorphism of $V_\rho$ such that $\tilde F w\tilde F^{-1}=F(w)$ on $V_\rho$ for all $w\in W_L$, 
\item $T_w\in \CT(L)$ is of class $[w]\in H^1(F, W_L)$, and $\chi_w := \chi|_{T_w}$. 
\end{itemize}
By Theorem \ref{MF}, for any Deligne-Lusztig characters $R^G_{T,\chi}$ and $R^G_{S,\eta}$ of $G^F$, we have the crude estimate 
\[
| m(R^G_{T,\chi}, R^G_{S,\eta}) |  \leq \sum_{\jmath\in J(S,T)} \frac{|W_G(T_\jmath)^F|}{|W_{G_\jmath}(T_\jmath)^F|} \\
 \leq |J(S, T)|\cdot n!  \leq 2^n n!.
\]
Assume that $\pi=R^G_{L, \chi, \rho}$ and $\sigma= R^G_{L', \chi', \rho'}$ as above. Since
$w\tilde F$ is a finite order automorphism of $V_\rho$, where $w\in W_L$, we have that
\[
|{\rm Tr}(w\tilde{F}, V_\rho)|\leq \dim\rho \leq \sqrt{|W_L|},
\]
and similar inequalities hold for $\rho'$. Thus
\[
m(\pi, \sigma) \leq 2^n n! \cdot \dim \rho \cdot \dim \rho' \leq 2^n n! \sqrt{|W_L| \cdot |W_{L'}|} \leq 2^n (n!)^2,
\]
noting that $|W_L|, |W_{L'}|\leq n!$. This finishes the proof. 
\end{proof}

\section{Two non-regular examples} \label{sec-TNRE}

In this section, we apply Theorem \ref{MF} to study two examples of Fourier-Jacobi models for certain non-regular Deligne-Lusztig characters. They arise from the depth zero case of  the local descent for $p$-adic unitary groups given in \cite{ST}, which will be studied in the next section.

For convenience, we switch to another standard notation for Deligne-Lusztig characters as follows. Let $G$ be a  connected reductive group  over $\Ff$ with Frobenius $F$. The pairs $(T, \chi)$ where 
$T\in \CT(G)$ and $\chi\in \Irr(T^F)$, are in natural correspondence with the pairs $(T^*, s)$ where $T^*$ is the dual torus of $T$ and $s\in T^{*, F}$. Then we write 
$R^G_{T, s} := R^G_{T,\chi}$. 

The algebraic subgroups in this section are all assumed to be $F$-stable. 

\subsection{Unitary groups}  \label{ssec-UG}
For $m \geq 1$, put  $\RG_m :=R_{\Ff_2/\Ff}\GL_m $, where 
$R_{\Ff_2/\Ff}$ denotes the Weil restriction of scalars for the quadratic extension $\Ff_2/\Ff$, so that $G_m^F = \GL_m(\Ff_2)$.  Let $\iota\in\Gal(\Ff_2/\Ff)$  be the nontrivial Galois automorphism.
A representation $\tau$ of  $\GL_m(\Ff_2)$ is called {\it conjugate self-dual} if $\tau^\vee\cong \tau^\iota$. 

Let $G=\RU_{4n-2}$ be defined over $\Ff$. Let $P$ be the Siegel parabolic subgroup of $G$, with Levi subgroup $M = \RG_{2n-1}$.   
Let $T\in \CT(M)$ be an elliptic torus (i.e. $T^F\cong \Ff_{4n-2}^\times$), and let $s\in T^{*, F}\cong \Ff_{4n-2}^\times$ be regular in $M^*$. Then $R^M_{T, s}\in \Irr(M^F)$ is  cuspidal, and all irreducible cuspidal representations of $M^F \cong \GL_{2n-1}(\Ff_2)$ arise in this way.
It is known that (e.g. \cite[Proposition 5.1]{Pr2}) $R^M_{T, s}$ is conjugate self-dual if and only if
\begin{equation} \label{SD}
 s\in \Ff_{4n-2}^1.
\end{equation} 
On the other hand, $\GL_m(\Ff_2)$ has no conjugate self-dual irreducible cuspidal representations if $m$ is odd. In this subsection, we assume that \eqref{SD} holds.

By induction in stages, 
\[
 \Ind^G_P ( R^M_{T, s} ) =  R^G_{T, s},
\]
where $\Ind^G_P$ denotes the parabolic induction from $P^F$ to $G^F$, and we have viewed $T\in \CT(G)$. The decomposition of the non-regular Deligne-Lusztig representation 
$R^G_{T, s}$ is as follows. 
Let $T_{\rm a}\in \CT(G)$ such that $T_{\rm a}^F\cong \Ff_{4n-2}^1\times \Ff_{4n-2}^1$ and view $s\in \Ff_{4n-2}^1$ as an element of  $T_{\rm a}^{*, F}$ via diagonal embedding. Put
\begin{equation}  \label{reg-ss}
 \pi_\reg := \frac{R^G_{T, s} - R^G_{T_{\rm a}, s}}{2},\quad  \pi_{\rm ss} :=  \frac{R^G_{T, s} + R^G_{T_{\rm a}, s}}{2}.
\end{equation} 
By (10.7.3) and (10.7.4) in \cite{DL},
\begin{equation} \label{Dec}
R^G_{T, s}=   \pi_\reg \oplus \pi_{\rm ss}. 
\end{equation} 
is the decomposition of $R^G_{T, s}$  into irreducible representations of $G^F$. Moreover $\pi_\reg$ is generic, and $\pi_{\rm ss}$ is non-generic. 

We briefly recall the Fourier-Jacobi model in this case, and refer to \cite{GGP1} for details. For $0\leq l \leq 2n-1$, let $P_l\subset P$ be the standard parabolic subgroup with Levi decomposition $P_l= M_l N_l$, such that $M_l \cong \RG_1^l \times H_l$, where $H_l:= \RU_{4n-2l-2}$. Recall that 
$\RG_1 =R_{\Ff_2/\Ff}\GL_1$.  Let $\pi$ and $\sigma$ be virtual characters of $G^F$ and $H^F_l$ respectively.  Denote by 
$\omega_{k,\psi}$, $k\geq 0$, the Weil representation of $\RU_k(\Ff)$ associated to $\psi$. The Fourier-Jacobi model of $\pi$ and $\sigma$ is as follows. 

The basic case that $l=0$ is given by \eqref{FJ-mult}. We recall it as 
\[
m(\pi,\sigma):=\langle \pi\otimes \omega_{4n-2, \psi}^\vee, \sigma \rangle_{G^F}.
\]
For $l>0$, one can define an $M_l$-equivariant  homomorphism 
\begin{equation} \label{ZH}
N_l \twoheadrightarrow Z_l \times \BH_{l},
\end{equation}
where $Z_l$ is a maximal unipotent  subgroup of $\RG_l$, and $\BH_{l}$ denotes the Heisenberg group associated to $H_l$. Let $\psi_l$ be a generic character of $Z_l^F$ defined using $\psi$, and put
\[
\nu_{l, \psi}: = \psi_l \otimes \omega_{4n-2l-2,\psi},
\]
which via \eqref{ZH} can be naturally viewed as a representation of $R_l^F$, with
\[
R_l:= H_l \ltimes N_l. 
\]
The Fourier-Jacobi model is concerned with the multiplicity 
\[
m(\pi,\sigma):= \langle \pi \otimes \nu_{l, \psi}^\vee, \sigma \rangle_{R_l^F}.
\]
In view of \cite{JZ}, for a representation $\pi$ of $G^F$, let $\CD_{l, \psi}(\pi)$ be the Jacquet module of $\pi \otimes \nu_{l, \psi}^\vee$ with respect to $N_l^F$, which is a representation of 
$H_l^F$, called the $l$-th {\it Fourier-Jacobi descent }of 
$\pi$ along $\psi$.  Then in this case we equivalently have that
\[
m(\pi,\sigma) = \langle \CD_{l, \psi}(\pi), \sigma\rangle_{H_l^F}. 
\]

For the example in this subsection, since the basic case that $l=0$ has been addressed by Theorem \ref{MF}, we only need to consider  the case that $0<l\leq 2n-1$.  Then we have the following  result, which in view of the decomposition  \eqref{Dec} can be interpreted as a finite field analog of the GGP conjecture for the Lusztig series of $s\in G^{*, F}$.  

\begin{thm} \label{U-des}
Let $R^G_{T, s}=\pi_\reg +\pi_{\rm ss}$ be given by \eqref{Dec}. Assume that $n\geq 2$. Then

\begin{enumerate}

\item For $0< l\leq 2n-1$ and any Deligne-Lusztig character $R^{H_l}_{S_0, s_0}$ of $H_l^F$, it holds that
\[
m(R^G_{T, s}, (-1)^{l+1+{\rm rk}\, S_0}R^{H_l}_{S_0, s_0})= 1.
\]

\item $\CD_{l, \psi}(\pi_{\rm ss})=0$ for $n-1<l\leq 2n-1$, and 
\begin{equation} \label{FJ-U}
 \CD_{n - 1, \psi}(\pi_{\rm ss}) = \bigoplus_{ a\in \Ff_2^1}(-1)^n R^{\RU_{2n}}_{S_0, ( - s, a)},
\end{equation}
where $S_0\in \CT(\RU_{2n})$ such that $S_0^F\cong \Ff_{4n-2}^1\times \Ff_2^1$.


\end{enumerate}
\end{thm}

To utilize Theorem \ref{MF}, we first reduce the Fourier-Jacobi model to the basic case as follows. 

\begin{lem} \label{red}
Let $\tau \in \Irr(\RG_{l}^F)$ be a cuspidal representation, where $0 < l<2n-1$. For $\sigma\in \Irr(H_l^F)$, put
\[
I(\tau, \sigma) = \Ind^G_{Q} (\tau\otimes\sigma),
\]
where $Q$ is a maximal  parabolic subgroup of $G$ with Levi subgroup $\RG_{l}\times H_{l}$. Then
\[
m(R^G_{T, s}, \sigma) = m(R^G_{T, s}, I(\tau, \sigma)). 
\]
\end{lem}

\begin{proof}
It is clear that $R^G_{T, s} = \Ind^G_P ( R^M_{T, s})$ does not belong to the Harish-Chandra series of $(\RG_{l}\times  L, \tau\otimes\mu)$, where $L$ is any Levi subgroup 
of $H_l$ and $\mu$ is any cuspidal representation of $L^F$. Thus the proof of \cite[Theorem 16.1]{GGP1} works verbatim in this case. 
\end{proof}

\begin{proof}(of Theorem \ref{U-des})
When $l=2n-1$,  it follows directly from the uniqueness of Whittaker models for standard modules that
\[
 \CD_{2n-1}(\pi_\reg)=\BC,\quad \CD_{2n-1}(\pi_{\rm ss})=0.
 \]
Below we assume that $0<l<2n-1$. 

We first prove (1).  Let $\tau\in \Irr(\RG^F_{l})$ be a cuspidal representation 
\[
\tau = (-1)^{l+1} R^{\RG_{l}}_{S_1, s_1}, 
\]
where $S_1\in\CT(\RG_{l})$,  $S_1^F\cong \Ff_{2l}^\times$, and $s_1\in S_1^{*, F}$ is regular in $\RG_{l}$. By induction in stages, 
\[
\Ind^G_Q (\tau\otimes  (-1)^{l+1+{\rm rk}\, S_0}R^{H_l}_{S_0, s_0}) =(-1)^{{\rm rk}\, S_0} R^G_{S, s'},
\]
where $Q$ is as in Lemma \ref{red}, $S = S_1\times S_0\in \CT(G)$ and $s'=(s_1, s_0)\in S^{*, F}$.
By  Lemma \ref{red}, it suffices to show that
\[
m(R^G_{T, s}, (-1)^{{\rm rk}\, S_0} R^G_{S, s'})=1.
\]
For this we apply Theorem \ref{MF}. Since $T$ and $S$ have no common rational  factors, $J(S, T)$ consists of a single element $\jmath$ for which $Z_\jmath$ is trivial and $G_\jmath= G$. Thus it is easy to see that
\[
m(R^G_{T, s}, R^G_{S, s'}) =e_{T, S}= (-1)^{{\rm rk}\, S_0},
\]
where $e_{T, S}$ is given by \eqref{eTS}. This proves Theorem \ref{U-des} (1).

Next we prove Theorem \ref{U-des} (2). First consider the case that $l=n-1$, and put
\[
\bar\sigma_{s,a}: = (-1)^n R^{\RU_{2n}}_{S_0, (-s, a)}\in \Irr(\RU_{2n}^F), 
\]
 where $S_0\in \CT(\RU_{2n})$ such that $S_0^F\cong \Ff_{4n-2}^1\times \Ff_2^1$, and $a \in \Ff_2^1$. Similar to the above, let $\tau = (-1)^n R^{\RG_{n-1}}_{S_1, s_1} \in \Irr(\RG^F_{n-1})$ be a cuspidal representation, so that by induction in stages
\[
I(\tau, \bar\sigma_{s,a}) =  R^G_{S, s'_a},
\]
where $S := S_1\times S_0\in \CT(G)$ and $s'_a := (s_1, -s, a)\in S^{*, F} \cong \Ff_{2n-2}^\times \times \Ff_{4n-2}^1 \times \Ff_2^1$.  By Lemma \ref{red} again, 
\[
m(\pi_{\rm ss}, \bar\sigma_{s,a})  = m(\pi_{\rm ss}, I(\tau, \bar\sigma_{s,a})) =  \frac{1}{2}\left(m(R^G_{T, s}, R^G_{S, s'_a})   + m(R^G_{T_{\rm a},s}, R^G_{S, s'_a})\right).
\]
We have seen that $m(R^G_{T, s}, R^G_{S, s'_a})  =1$. Again we apply  Theorem \ref{MF}  to evaluate $ m(R^G_{T_{\rm a},s}, R^G_{S, s'_a})$.
We have $J(S, T_{\rm a})=\{\jmath_0, \jmath_1\}$, for which $Z_{\jmath_0}=\{1\}$, $G_{\jmath_0}=G$ and 
\[
Z_{\jmath_1}^F\cong \Ff_{4n-2}^1,\quad G_{\jmath_1} \cong Z_{\jmath_1} \times \RU_{2n-1}.
\] 
Denote the $\jmath$-summand in Theorem \ref{MF} by $m_\jmath$ for short. Similar to the above, we see that
\[
m_{\jmath_0} = e_{T_{\rm a}, S}= -1.
\]
To compute $m_{\jmath_1}$, we may conjugate $T_{\rm a}$ into $G_{\jmath_1}$, and  assume that $Z_{\jmath_1}^F$ is the first factor of 
$T_{\rm a}^F \cong \Ff_{4n-2}^1\times \Ff_{4n-2}^1$. Recall that $s$ embeds into $T_{\rm a}^{*, F}$ diagonally as $(s, s)$.  We find that
\[
\left|\Set{w\in W_G(T_{\rm a})^F | w(s, s)|_{Z_{\jmath_1}^F} =s} \right| = 2(2n-1) = 2|W_{G_{\jmath_1}}(T_{\rm a})^F|.
\]
From this it follows easily that
\[
m_{\jmath_1} = e_{T_{\rm a}, S} (-1)^{l(Z_{\jmath_1})}\cdot 2 =2. 
\]
Thus
\[
m(R^G_{T_{\rm a}, s}, R^G_{S, s'_a}) =m_{\jmath_0} + m_{\jmath_1} =1. 
\]
This proves that $m(\pi_{\rm ss}, \bar\sigma_{s,a})=1$.

Note that the summand in \eqref{FJ-U} corresponding to $a\in \Ff_2^1$, which we denote by 
$\bar\sigma_{s,a}$ in the above, is the unique member of its own Lusztig series (see \cite{L2, L3} for this notion).  Thus it remains to show that for any  character of $H_l^F$ of the form
 \[
 \sigma = (-1)^{l+1+{\rm rk}\, S_0}R^{H_l}_{S_0, s_0},\quad S_0\in \CT(H_l), \quad s_0 \in S_0^{*, F},
 \]
where $n-1\leq l< 2n-1$ and $\sigma$ is not isomorphic to any summand $\bar\sigma_{s,a}$ in \eqref{FJ-U} when $l=n-1$, it holds that
\[
m(\pi_{\rm ss}, \sigma) = 0.
\]
Applying  Lemma \ref{red} again, it is reduced to proving that
\begin{equation} \label{Zero}
m(\pi_{\rm ss}, R^G_{S, s'})=0,
\end{equation} 
where $S = S_1\times S_0$, $S_1\in \CT(\RG_l)$, $S_1^F\cong \Ff_{2l}^\times$, $s_1\in S_1^{*, F}$ and $s'=(s_1, s_0)\in S^{*, F}$. Using  Theorem \ref{MF}, similar calculations show that in this case
\[
m(R^G_{T, s}, R^G_{S, s'}) = e_{T, S}  = -e_{T_{\rm a}, S} = - m(R^G_{T_{\rm a}, s}, R^G_{S, s'}),
\]
which implies \eqref{Zero}, hence proves Theorem \ref{U-des} (2). 
\end{proof}

\subsection{Symplectic groups} \label{ssec6.2} For the second example, we let $G=\Sp_{4n} $ be defined over $\Ff$. Let $P$ be the Siegel parabolic subgroup of $G$ with Levi subgroup
$M = \GL_{2n}$. Let $L\cong \GL_n\times \GL_n$ be a Levi subgroup of $M$. A representation $\tau$ of $M^F$ is said to have a {\it linear model} if 
\[
\Hom_{L^F}(\tau, \mathbbm{1})\neq 0. 
\]
It is known that a cuspidal representation $\tau\in \Irr(M^F)$  has a linear model if and only if $\tau$ is self-dual, i.e. $\tau\cong\tau^\vee$. 
   Let $T \in \CT(M)$ be elliptic (i.e. $T^F\cong \Ff_{2n}^\times$), and let $s\in T^{*,F}$ be regular in $M^*$, so that $ - R^M_{T, s} \in \Irr(M^F)$ is cuspidal. By \cite{L4} (see also \cite{H1}), $-R^M_{T, s}$ has a  linear model if and only if 
\begin{equation} \label{LM}
s \in \Ff_{2n}^1.
\end{equation} 
In this subsection, we assume that \eqref{LM} holds. We mention that, in \cite{LMNW} a criterion for the existence of  linear models is  given for general irreducible Deligne-Lusztig characters of $M^F$ which are not necessarily cuspidal.  

Let $T_{\rm a}\in \CT(G)$ such that 
$T_{\rm a}^F\cong \Ff_{2n}^1\times \Ff_{2n}^1$, and embed $s$ into $T_{\rm a}$ diagonally. We have 
\begin{equation} \label{Dec-SP}
\Ind^G_P(-R^M_{T, s}) = -R^G_{T, s} =\pi_\reg \oplus  \pi_{\rm ss},
\end{equation}
where
\begin{equation} \label{reg-ss-SP}
 \pi_\reg :=  \frac{-R^G_{T,s} + R^G_{T_{\rm a}, s}}{2},\quad \textrm{and}\quad \pi_{\rm ss} :=  \frac{-R^G_{T,s} - R^G_{T_{\rm a}, s}}{2}
\end{equation} 
are irreducible representations of $G^F$. Again $\pi_\reg$ is generic and $\pi_{\rm ss}$ is not. 

For $0\leq l\leq 2n$, let $P_l=M_l N_l\subset P$ be the standard parabolic subgroup of $G$ with Levi subgroup $M_l \cong \GL_1^l \times H_l$, where $H_l :=\Sp_{4n-2l}$. Similar to the  case of unitary groups, one can define a representation $\nu_{l, \psi}$ of 
\[
R_l:= H_l \ltimes N_l,
\]
which is the Weil representation $\omega_{n,\psi}$ of $G^F$ in the case that $l=0$. For virtual characters $\pi$ and $\sigma$ of $G^F$ and $H_l^F$ respectively, one defines the Fourier-Jacobi model
\[
m(\pi,\sigma):= \langle \pi\otimes \nu_{l,\psi}^\vee, \sigma\rangle_{R_l^F}. 
\]
Again we refer to \cite{GGP1} for details. If $\pi$ is a representation of $G^F$, the $l$-th Fourier-Jacobi descent of $\pi$ along $\psi$ is Jacquet module of $\pi\otimes \nu_{l,\psi}^\vee$ with respect to $N_l^F$, which is a representation of $H_l^F$, denoted by $\CD_{l,\psi}(\pi)$. In  this case $m(\pi, \sigma) = \langle \CD_{l, \psi}(\pi),\sigma\rangle_{H_l^F}$. 

We have the following result for this example.  
 
 \begin{thm} \label{SP-des}
 Let $-R^G_{T, s} = \pi_\reg + \pi_{\rm ss}$ be given by \eqref{Dec-SP}.  Then
 
 \begin{enumerate}
\item For $0< l\leq 2n$ and any Deligne-Lusztig character $R^{H_l}_{S_0, s_0}$ of $H_l^F$, it holds that
\[
m(-R^G_{T, s}, (-1)^{l+{\rm rk}\, S_0}R^{H_l}_{S_0, s_0})= 1.
\] 

\item $\CD_{l, \psi}(\pi_{\rm ss})=0$ for $n<l\leq 2n$, and 
\begin{equation} \label{FJ-SP}
 \CD_{n , \psi}(\pi_{\rm ss}) = (-1)^n R^{\Sp_{2n}}_{S_0,  - s},
\end{equation}
where $S_0\in \CT(\Sp_{2n})$ such that $S_0^F\cong \Ff_{2n}^1$. 
 \end{enumerate}
 \end{thm}
 
 The proof of Theorem \ref{SP-des} is similar to that of Theorem \ref{U-des}. One can first reduce  to the basic case using the following lemma, and then apply Theorem
 \ref{MF}. The details will be omitted. 
 
 \begin{lem} \label{red-SP}
Let $\tau \in \Irr(\GL_{l}^F)$ be a cuspidal representation, where $0 < l<2n$. For $\sigma\in \Irr(H_l^F)$, put
\[
I(\tau, \sigma) = \Ind^G_{Q} (\tau\otimes\sigma),
\]
where $Q$ is a maximal  parabolic subgroup of $G$ with Levi subgroup $\GL_{l}\times H_{l}$. Then
\[
m(-R^G_{T, s}, \sigma) = m(-R^G_{T, s}, I(\tau, \sigma)). 
\]
\end{lem}

\section{Local descent for $p$-adic unitary groups}  \label{sec-LDPUG}

In the rest of this paper, we consider unitary groups over $p$-adic local fields. From now on we  systematically change the notations and write ${\sf G}, {\sf T},  {\sf U}_n, {\sf GL}_n$ etc. for algebraic groups over a finite field $\Ff$, and reserve  $G, T, \RU_n, \GL_n$ etc. for $p$-adic groups. 

This section is basically expository. Let $\RE /\RF$ be a quadratic extension of $p$-adic local fields, with  odd residue characteristic $p$. Let $\iota \in \Gal(\RE/\RF)$ be the nontrivial Galois automorphism.  In this section we first collect some  results on irreducible representations of $\GL_m(\RE)$ distinguished by the Galois involution, namely, by the subgroup $\GL_m(\RE)^\iota = \GL_m(\RF)$. Then we recall  the local descent of irreducible distinguished supercuspidal representations to unitary groups from \cite{ST}, and explain the connections with the Langlands functoriality and the non-tempered local GGP conjecture \cite{GGP2}. Finally we give the distinction criterion in the depth zero case following \cite{CG, HM1}.

\subsection{Galois distinction} 

Let $\omega_{\RE/\RF}$ be the quadratic character of $\RF^\times$ given by the local class field theory, and fix a character $\mu$ of $\RE^\times$ such that $\mu|_{\RF^\times} = \omega_{\RE/\RF}$. By composition with determinant maps, we regard $\omega_{\RE/\RF}$ and $\mu$ as characters of $\GL_m(\RF)$ and $\GL_m(\RE)$ respectively.  Let $\tau$ be an irreducible admissible representation $\tau$ of $\GL_m(\RE)$. We are concerned with the following notions: 
\begin{itemize}
\item $\tau$ is called {\it distinguished} (resp. {\it $\omega_{\RE/\RF}$-distinguished}) if 
\[
\Hom_{\GL_m(\RF)}(\tau, \mathbbm{1})\neq 0 \quad (\textrm{resp. } \Hom_{\GL_m(\RF)}(\tau, \omega_{\RE/\RF})\neq 0),
\]
in which case the Hom space has dimension one. 
Note that $\tau$ is $\omega_{\RE/\RF}$-distinguished if and only if $\tau\otimes\mu$ is distinguished. 
\item $\tau$ is called {\it conjugate self-dual} if $\tau^\vee\cong \tau^\iota$. 
\end{itemize}
It is well-known that  (e.g. \cite{F, Pr1}) if $\tau$ is distinguished, then $\tau$ is conjugate self-dual and the central character of $\tau$ is trivial on $\RF^\times$.  

Let $\tau$ be an irreducible square-integrable conjugate self-dual representation of $\GL_m(\RE)$, whose central character is trivial on $\RF^\times$. By \cite{K, AKT}, if  $m$ is odd,  then $\tau$ is distinguished; if $m$ is even, then exactly one of $\tau$ and $\tau\otimes\mu$ is distinguished. We have that
\[
L({\sfs}, \tau\times \tau^\iota) = L( {\sfs}, \tau, {\rm As}) L({\sfs}, \tau\otimes\mu, {\rm As}),
\]
where the left hand side is the Rankin-Selberg L-function, and factors in the right hand side are the Asai L-functions, with ${\rm As} ={\rm As}^+$.  See e.g. \cite[\S7]{GGP1} for the definition of two Asai representations ${\rm As}^+$ and ${\rm As}^-$, which correspond to the stable and unstable base change  ${}^L \RU_m \to {}^L \GL_m$ respectively (\cite{F, KK}). It follows from \cite{K} that $L( {\sfs}, \tau, {\rm As})$ has a pole at ${\sfs} =0$ if and only if $\tau$ is distinguished. 

\subsection{Local descent and non-tempered GGP conjecture} \label{ssec-LDGGP}
Assume from now on that $\tau$ is irreducible distinguished and is moreover supercuspidal. Then $\tau$ (resp. $\tau\otimes\mu^{-1}$) is a stable base change of an irreducible generic supercuspidal representation of a quasi-split $\RU_m(\RF)$ for $m$ odd (resp. even). Indeed, let $M_\tau$ be the L-parameter of $\tau$ under the local Langlands correspondence, which is an $m$-dimensional  irreducible  representation of the Weil group $W_\RE$, extended trivially to the Weil-Deligne group $WD_\RE = W_\RE\times \SL_2(\BC)$. Then ${\rm As}^+(M_\tau)^{WD_\RE}\neq 0$.  By 
\cite[Prop. 7.5]{GGP1}, $M_\tau$ is conjugate orthogonal and  $M_\tau(\mu^{-1}):=M_\tau\otimes\mu^{-1}$ is conjugate symplectic, which give an L-parameter of $\RU_m(\RF)$ for $m$ odd and even respectively. Here $\mu$ is viewed as a character of $W_\RE$ via the local class field theory.

From the representation $\tau$ of $\GL_m(\RE)$, the Fourier-Jacobi local descent  construction in \cite{ST} produces generic supercuspidal representations of $\RU_{2n}(\RF)$, where $n:=\lfloor\frac{m+1}{2}\rfloor$. In particular, it realizes the inverse map to the above base change for the representation $\tau^\vee \otimes \mu$ when $m=2n$ is even. We shall briefly recall the construction below, with slightly different notation and convention.  

Let $P$ be the Siegel parabolic subgroup of $\RU_{2m}(\RF)$ with Levi subgroup $M \cong \GL_m(\RE)$. For ${\sfs}\in\BC$, define the normalized parabolic induction 
\begin{equation} \label{PI}
\rho_{\tau, {\sfs}} := \Ind^{\RU_{2m}(\RF)}_P \tau_{\sfs},\quad \tau_{\sfs}:= \tau\otimes |\det|_\RE^{{\sfs}-1/2},
\end{equation}
where $|\cdot|_\RE$ is the normalized absolute value on $\RE$. Then $\rho_{\tau, 1}$ is of length two, and we have a short exact sequence 
\begin{equation}  \label{SES}
 0 \longrightarrow \rho_\tau  \longrightarrow   \rho_{\tau, 1}  \longrightarrow  \pi_\tau  \longrightarrow  0
\end{equation} 
where the subrepresentation $\rho_\tau$ of $\rho_{\tau, 1}$ is generic, and the Langlands quotient $\pi_\tau$ of $\rho_{\tau, 1}$ is non-generic. 

Fix a nontrivial additive character $\psi_\RF$ of $\RF$. 
For $0\leq l\leq m$, let $P_l=M_l N_l \subset P$ be the standard parabolic subgroup of $\RU_{2m}(\RF)$ with Levi subgroup $M_l \cong (\RE^\times)^l\times \RU_{2m-2l}(\RF)$. Similar to the finite field case, for $l>0$ we have 
a homomorphism 
$
N_l \twoheadrightarrow Z_l \times \BH_l,
$
where $Z_l$ is a maximal unipotent subgroup of $\GL_l(\RE)$ and $\BH_l$ is the Heisenberg group associated to $\RU_{2m-2l}$. 
 As in \cite{GGP1}, associated to 
 $\psi_\RF$ and the character $\mu$, one has an irreducible unitary representation 
 \[
 \nu_{l, \psi_\RF, \mu} = \psi_l \otimes \omega_{2m-2l, \psi_\RF, \mu}
 \]
 of $R_l:=\RU_{2m-2l}(\RF)\ltimes N_l$. Here $\psi_l$ is the generic character of $Z_l$ defined using $\psi_\RF$, and $\omega_{2m-2l, \psi_\RF, \mu}$ is the Weil representation 
 of $\RU_{2m-2l}(\RF)\ltimes \BH_l$ defined using $\psi_\RF$ and $\mu$. 
 
 \begin{defn}
 The  $l$-th  Fourier-Jacobi local descent of a smooth representation $\pi$ of $\RU_{2m}(\RF)$,  denoted by $\CD_{l, \psi_\RF, \mu}(\pi)$, is  a smooth representation of $\RU_{2m-2l}(\RF)$ defined as the Jacquet module of $\pi\otimes \nu_{l, \psi_\RF, \mu}^\vee$ with respect to $N_l$.
 \end{defn}

\begin{rmk}  \label{ST-rmk}
Our formulation follows  {\rm \cite{GGP1}}, and differs from {\rm \cite{ST}} by taking the contragredient of the Weil representation. This can be easily rectified as follows. 
\begin{enumerate}
\item
The Whittaker datum for $\RU_{2m-2l}(\RF)$, as well as the Weil representation $\omega_{2m-2l, \psi_\RF, \mu}$ {\rm (}for fixed $\mu${\rm)}, is determined by the choice of $\psi_\RF$ up to a twist by an element of $N_{\RE/\RF}(\RE^\times)$. 
\item We can use   that $\omega_{2m-2l, \psi_\RF, \mu}^\vee \cong \omega_{2m-2l, \psi_\RF^{-1}, \mu^{-1}}$ to translate the  results from {\rm \cite{ST}} accordingly. 
\end{enumerate}
\end{rmk}

In view of Remark \ref{ST-rmk}, we can state the following

\begin{thm}[\cite{ST}] \label{ST-thm}
Let $\tau$ be an irreducible supercuspidal representation of $\GL_m(\RE)$ such that $L({\sfs}, \tau, {\rm As})$ has a pole at ${\sfs} =0$. Put $n=\lfloor\frac{m+1}{2}\rfloor$ and $l_0=m-n =\lfloor \frac{m}{2} \rfloor$. Then the following hold.
\begin{enumerate}
\item $\CD_{l, \psi_\RF, \mu}(\pi_\tau)=0$ for $l>l_0$, and $\CD_{l_0, \psi_\RF, \mu}(\pi_\tau)\neq 0$ is the multiplicity free direct sum of irreducible $\psi_\RF^{-1}$-generic supercuspidal representations $\sigma$ of $\RU_{2n}(\RF)$ such that the local gamma factor $\gamma(s, \sigma^\vee\times (\tau\otimes\mu^{-1}), \psi_\RF^{-1})$ has a pole at $s=1$.

\item If $m=2n$ is even, then $\CD_{n, \psi_\RF, \mu}(\pi_\tau)$ is irreducible. 
\end{enumerate}
\end{thm}

This result is compatible with the non-tempered local GGP conjecture \cite{GGP2}. Recall that the L-parameter $M_\tau$ of $\tau$ is conjugate orthogonal. The A-parameter 
of $\pi_\tau$ is the representation $M_\tau \otimes [2]$ of $WD_\RE\times\SL_2(\BC)$, where $[k]$, $k\geq 1$, denotes the irreducible algebraic $k$-dimensional representation of the Arthur $\SL_2(\BC)$. Note that the L-parameter associated to $M_\tau \otimes [2]$  is $M_\tau \nu^{1/2} \oplus M_\tau \nu^{-1/2}$, where  $\nu =|\det|_\RE$.   The non-tempered GGP conjecture  \cite[Conjecture 6.1]{GGP2} predicts that for any irreducible quotient $\sigma$ of $\CD_{l_0}(\pi_\tau)$, the A-parameter  $N_\sigma$ of $\sigma$ and $M_\tau (\mu^{-1})\otimes [2]$ form a {\it relevant pair}. 

By \Cref{ST-thm}, we have a well defined map  
\begin{equation}\label{eq:descent}
\begin{array}{rcl}
   \DD\colon \Set{\tau\in \Irr(\GL_m(E)) | \begin{array}{l} \text{$\tau$ is supercuspidal and} \\
   \text{$L(\sfs, \tau,\Asai)$ has a pole at $\sfs=0$}
   \end{array}} & \xrightarrow{\ \ \ } & \mathrm{Rep}(\RU_{2\lfloor\frac{m+1}{2} \rfloor}(\RF)),\\
  \tau & \mapsto & \CD_{\lfloor{\frac{m}{2}}\rfloor,  \psi_\RF,\mu}(\pi_\tau), 
  \end{array}
\end{equation}
and the image of $\DD$ consists of certain semisimple representations of 
$\RU_{2\lfloor\frac{m+1}{2} \rfloor}(\RF)$. 
By the definition of the relevant pair of A-parameters \cite[\S3]{GGP2}, the following are clear:
\begin{itemize}
\item If $m=2n-1$ is odd, then $N_\sigma$ is of the form $M_\tau (\mu^{-1}) \oplus \zeta$, where $\zeta$ is  a conjugate symplectic character of $\RE^\times$ (i.e. $\zeta|_{\RF^\times} = \omega_{\RE/\RF}$);
\item If $m=2n$ is even, then $N_\sigma = M_\tau (\mu^{-1})$. 
\end{itemize}

\subsection{Distinction criterion: depth zero case}
The general distinction problem  for supercuspidal representations has been studied extensively (see e.g. \cite{HM2, H1, H2, Z}). For the Galois distinction of our concern, a criterion for tame supercuspidal representations is given in \cite{HM1}, and the  general case is given in \cite{S}.  In the depth zero case, we shall follow the result in \cite{CG}, which gives a necessary and sufficient condition for the Galois distinction. It recovers the depth zero case of the criterion in \cite{HM1}.

An irreducible depth zero supercuspidal representation $\tau$ of $\GL_m(\RE)$ is regular in the sense of \cite{Ka}, and is associated to a pair $(T, \chi)$, where $T\cong \RE_m^\times$ is an elliptic torus in $\GL_m(\RE)$ and $\chi$ is a regular depth zero character of $T$.  Here $\RE_m$ denotes the  degree $m$ unramified extension of $\RE$. Then we write 
$\tau = \tau_{(T, \chi)}$.

\begin{thm}[\cite{CG, HM1}] \label{CG-thm}
An irreducible depth zero supercuspidal representation $\tau$ of $\GL_m(\RE)$
is  distinguished if and only if  $\tau = \tau_{(T,\chi)}$ satisfies the conditions in one of the following two cases:
\begin{enumerate}
\item $\RE/\RF$ is unramified, $m=2n-1$ is odd, and $\chi|_{\RF^\times_{2n-1}} =\mathbbm{1} $.
\item $\RE/\RF$ is ramified, $m=2n$ is even, and  $\chi|_{\RF_{2n}^\times} = \omega_{\RE_{2n}/ \RF_{2n}}$. \end{enumerate}
\end{thm}

\begin{rmk}
Theorem \ref{CG-thm} indicates that the sufficient condition for the Galois distinction given in {\rm \cite[Theorem 1.1]{HM1}} is also necessary in the depth zero case. Note that in both cases {\rm (1)} and {\rm (2)} of Theorem \ref{CG-thm},
$\RE\cap \RF_m =\RF$ and the nontrivial Galois automorphism $\iota_m \in\Gal(\RE_m/\RF_m)$ restricts to $\iota\in \Gal(\RE/\RF):$
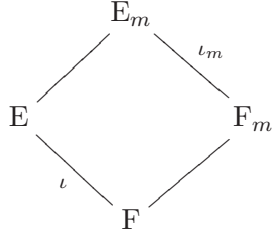
\[
\xymatrix{
& \RE_m   \ar@{-}[dl] \ar@{-}[dr]^{\iota_m}  & \\  \RE \ar@{-}[dr]_{\iota} & &  \RF_m \ar@{-}[dl] \\ & \RF & 
}
\]
\end{rmk}

To compute the local descent, we reformulate \Cref{CG-thm} in terms of compact induction. Let $\Fe$ and $\Ff$ be the residue fields of $\RE$ and $\RF$ respectively, so that $\Fe= \Ff_2$ or $\Ff$ according to whether $\RE/\RF$ is unramified or not. 
Let $M:= \GL_m(\RE)$ and denote by $\CB(M)$ the building of $M$. For a hyperspecial point $x\in \CB(M)$, denote by $[x]$ the image of $x$ in the reduced building
 of $M$.  Let $M_x$ and $M_{[x]}$ be the stabilizer of $x$ and $[x]$ respectively. Then 
\[
M_{[x]}= Z  M_x = \braket{\varpi_E}\times M_x,
\]
where $Z\cong \RE^\times$ is the center of $M$ and $\varpi_E$ is a fixed uniformizer of $E$. 
The reductive quotient $M_x/M_{x,0+} $ of $M_x \cong \GL_m(\CO_\RE)$ is naturally isomorphic to  
\[
{\sf M}^F := {\sf GL}_m(\Fe),
\]
where 
${\sf M} := R_{\Fe/\Ff}{\sf GL}_m$ is defined over $\Ff$, and $\CO_\RE$ is the ring of integers of $\RE$.
See \cite{MP} for the notion of Moy-Prasad filtration.
Let $\sfT$ be an maximally anisotropic maximal tori in $\sfM$ so that $\sfT^F$ and ${\sfT^*}^F$ are naturally identified with  $\Fe_m$.  
Let $\Femo:= \Set{a\in \Fe_m| \Fe(a) = \Fe_m}$ denote the set of regular elements in $\Fe_m$. 
The Galois group $\Gal(\Fe_m/\Fe)$ acts on $\Femo$.   
For $s\in \Femo$ and $c \in \bC^\times$, 
let $\bartau_{s,\cc}$ be the irreducible representation of $M_{[x]}$ such that 
\begin{itemize}
\item
$\bartau_{s,\cc}|_{M_x}$is the inflation of Deligne-Lusztig character $(-1)^{m+1}R_{\sfT,s}^{\sfM}$,
\item and  $\bartau_{s,\cc}(\varpi_E) = \cc$.
\end{itemize}
The following map is a bijection that gives the well-known classification of irreducible depth zero supercuspidal representation of $M$ (see e.g. \cite{BK}):
\[
\begin{array}{rcl}
\big(\Gal(\Fe_m/\Fe) \backslash \Femo \big)\times \bC^\times & \xrightarrow{\ \ \ \ } & \Set{\pi \in \Irr(\GL_m(E))| \text{$\pi$ is supercuspidual and depth zero}} \\
(\Gal(\Fe_m/\Fe)\cdot s,\cc) &\mapsto &  \tau_{s,\cc}:= \cInd_{M_{[x]}}^M \bartau_{s,\cc}. 
\end{array}
\]
%
Let 
\[
\Femoo:=\Femo\cap \Fe_m^1. 
\]

The following result is equivalent to \Cref{CG-thm} and can be easily deduced from the main theorem 
of \cite{CG}.

\begin{thm} \label{thm-dist}
Suppose there is an irreducible depth zero supercuspidal representation  of $M=\GL_m(\RE)$ distinguished by 
$\GL_m(\RF)$. Then  $\RE/\RF$ is unramified if $m$ is odd, and $\RE/\RF$ is ramified if $m$ is even. 
Moreover, there is a bijection
\[
    \begin{array}{rcl}
    \Gal(\Fe_m/\Fe) \backslash\, \Femoo & \xrightarrow{\ \ \ } & 
    \Set{\pi \in \Irr(\GL_m(\RE)) | \begin{array}{l}\text{$\pi$ is supercuspidal, depth zero} \\
    \text{and distinguished by $\GL_m(\RF)$}
    \end{array}}, \\
    s& \mapsto & \tau_s := \tau_{s,\cc_s},
    \end{array}
\]
where 
\[
\cc_s := \begin{cases} 1  & \text{if $\RE/\RF$ is unramified ,}\\
-\vartheta'_n(s) & \text{if $\RE/\RF$ is ramified.}
\end{cases}
\]
\end{thm}

\delete{
\begin{thm} \label{thm-dist}
An irreducible depth zero supercuspidal representation $\tau$ of $M=\GL_m(\RE)$ is distinguished if and only if $\tau = \textrm{c-}\Ind^M_{M_{[x]}} \bar{\tau}$ with $x$ a hyperspecial point of $\CB(M)$ and $\bar\tau|_{M_x}$ is inflated from an irreducible cuspidal representation $\bar\tau'$ of ${\sf M}^F$ such that  one of the following hold:
\begin{enumerate}
\item  $\RE/\RF$ is unramified, $m=2n-1$ is odd, ${\sf M}^F= {\sf GL}_{m}(\Ff_2)$,  $\bar{\tau}'$ is conjugate self-dual as in Section \ref{ssec-UG}:
\[
\bar\tau' = R^{{\sf M}}_{{\sf T}, s},\quad {\sf T}^F\cong \Ff_{2m}^\times, \quad s\in \Ff_{2m}^1\textrm{ regular in }{\sf M}^*,
\]
and $\chi_\tau(\varpi_\RF)=1$, where $\varpi_\RF$ is a uniformizer of $\RF$.

\item $\RE/\RF$ is ramified, $m=2n$ is even,  ${\sf M}^F= {\sf GL}_{m}(\Ff)$,  $\bar\tau'$ is self-dual as in Section \ref{ssec6.2}:
\[
\bar\tau' = -R^{{\sf M}}_{{\sf T}, s},\quad {\sf T}^F\cong \Ff_{m}^\times, \quad s\in \Ff_{m}^1\textrm{ regular in }{\sf M}^*,
\]
and $\chi_\tau (\varpi_\RE) = -\vartheta'_n(s)$, where $\varpi_\RE $ is a uniformizer of $\RE$ such that $\iota(\varpi_\RE)=-\varpi_\RE$ and $\vartheta'_n$ is the unique nontrivial quadratic character of $\Ff_{m}^1$ (see \Cref{ssec2.3}).
\end{enumerate}
\end{thm}

}

\section{Models and splittings of Weil representations} \label{sec-MSWR}

 Let $\Fp_\RE$ and $\Fp_\RF$ be the maximal ideals of $\CO_\RE$ and $\CO_\RF$ respectively. Assume that 
the additive character $\psi_\RF$ of $\RF$ has conductor $\Fp_\RF$, so that $\psi_\RF|_{\CO_\RF}$ is inflated from a nontrivial additive character $\psi_\Ff$ of $\Ff$. 
Let $W =\RE^{2n}$ be  endowed with a skew-Hermitian form $\langle, \rangle$ represented by the matrix 
\[
J_{2n} = \begin{pmatrix} & w_n \\ -w_n & \end{pmatrix},\quad{\rm with}\quad w_n = \begin{pmatrix} & & & 1 \\ & & 1 & \\
& \cdots & & \\
1 & & & \end{pmatrix}_{n\times n}.
\]
Then we have the realization 
\[
\RU_{2n}(\RF) =\RU(W) =  \Set{g\in \GL_{2n}(\RE)  | gJ_{2n}  \iota(g{}^t) = J_{2n}},
\]
where $g^t$ denotes the transpose of $g$. 
Let  $L = \CO_\RE^{2n}$, which is a self-dual lattice of $V$ in the sense of \eqref{eq:sdlattice}.
Put $L^\sharp= \Fp_\RE L$. Then $L/L^\sharp = \Fe^{2n}$ is a skew-Hermitian or symplectic space according to $\Fe=\Ff_2$ or $\Ff$, whose isometry group
\begin{equation} \label{fH}
{\sf H} = {\sf U}_{2n}  \quad {\rm or } \quad {\sf Sp}_{2n}
\end{equation}
is defined over $\Ff$. The stabilizer of $L$ in $\RU(W)$,
\[
\RU(W)_L := \{g\in \RU(W) : gL=L\},
\]
is a maximal compact subgroup of $\RU(W)$ with finite reductive quotient ${\sf H}^F$.

Let $\BH_W := W\times \RF$ be the Heisenberg group associated to $V$. By the Stone-von Neumann Theorem, up to isomorphism, there is a unique irreducible smooth representation 
  $\rho_{\psi_\RF}$ of $\BH_W$ with central character $\psi_\RF$. We consider two realizations  of $\rho_{\psi_\RF}$. 
  
  The {\it Schr\"odinger model}
of $\rho_{\psi_\RF}$ is realized on the space $\CS(Y)$ of locally constant, compactly supported functions on a Lagrangian subspace $Y$ of $V$. Following \cite{RR}, there is a projective representation 
\[
\RU(W)\to \GL(\CS(Y)),\quad g\mapsto M^Y_g
\]
of $\RU(W)$, with Ranga-Rao cocycle $c_Y(g, g')$. Given a character $\mu$ of $\RE^\times$ satisfying that $\mu|_{\RF^\times} = \omega_{\RE/\RF}$, a splitting
\[
\beta^Y_\mu: \RU(W)\to \BC^\times
\]
of $c_Y(g, g')$ is constructed in \cite{Ku}. This gives the Weil representation $\omega_{\psi_\RF, \mu} := \omega_{2n, \psi_{\RF}, \mu}$ of $\RU(W)$  on $\CS(Y)$, so that $g\in \RU(W)$ acts by $ \beta^Y_\mu(g) M^Y_g $.

Let $(\rho_{\psi_\Ff}, S)$ be the unique irreducible representation of $\BH_{L/L^\sharp}:=L/L^\sharp \times \Ff$ with central character $\psi_\Ff$, inflated to  a representation 
$\bar{\rho}_{\psi_\RF}$ of $\BH_L := L \times \CO_\RF$. The {\it generalized lattice model} of $\rho_{\psi_\RF}$ is realized on the space $\CS(L)$ of locally constant, compactly supported maps
$f: W\to S$ satisfying that
\begin{equation} \label{rho}
f(v+x ) = \psi_\RF\big(\frac{1}{2}{\rm tr}_{\RE/\RF}( \langle x, v\rangle)\big) \bar{\rho}_{\psi_\RF}(v). f(x),\quad v\in L, \quad x\in W.
\end{equation}
By \cite[Theorem 3.4]{P1}, there is a projective representation 
\[
\RU(W)\to \GL(\CS(L)),\quad g\mapsto M^L_g,
\]
with a unique splitting 
$
\beta^L: \RU(W)\to \BC^\times
$
such that $\beta^L|_{\RU(W)_L}\equiv 1$. Then  $g\in \RU(W)$ acts by $\beta^L(g)M^L_g$  on $\CS(L)$ as a genuine representation.  Recall that 
$S$ supplies the Weil representation $\omega_{\psi_\Ff}$ of ${\sf H}^F$, and we inflate $\omega_{\psi_\Ff}$ to a representation $\bar{\omega}_{\psi_\RF}$ of $\RU(W)_L$. Following \cite{P1}, 
\begin{equation} \label{ML}
(M^L_k. f)(x) = \bar{\omega}_{\psi_\RF}(k). f(k^{-1}x)
\end{equation} 
for $k\in \RU(W)_L$, $f\in \CS(L)$ and $x\in W$.

Let $\Psi: \CS(L) \to \CS(Y)$ be an isomorphism as irreducible representations of $\BH_V$, which is unique up to scalar.  Define the function 
\[
\alpha: \RU(W)\to \BC^\times
\]
such that $\alpha(g)\cdot \Psi\circ M^L_g = M^Y_g\circ \Psi$, $g\in \RU(W)$, and define $\xi_\mu$ to be the ratio of the splitting $\alpha \beta^Y_\mu$ to Pan's splitting $\beta^L$:
\begin{equation}\label{eq:xi.mu}
\xi_\mu := \alpha \beta^Y_\mu (\beta^L)^{-1}: \RU(W)\to \BC^\times.
\end{equation}
Thus we have a commutative diagram 
\[
\xymatrix{
\CS(L) \ar[r]^\Psi \ar[d]_{\xi_\mu(g)\beta^L(g)M^L_g} &  \CS(Y) \ar[d]^{\beta^Y_\mu(g)M^Y_g} \\
\CS(L) \ar[r]^\Psi & \CS(Y)
}
\]
for all $g\in \RU(W)$. By \cite{P1, O}, $\xi_\mu |_{\RU(W)_L}$ is a character that factors through the determinant. That is, 
\[
\xi_\mu|_{\RU(W)_L} = \xi_\mu'\circ \det: \RU(W)_L \to \BC^\times
\]
for a character $\xi_\mu'$ of $\RE_L:=\det(\RU(W)_L)$. We summary above discussions as follows.

\begin{prop} \label{gen-latt}
The Weil representation $\omega_{\psi_\RF, \mu}$ of $\RU(W)$ can be realized on the generalized lattice model $\CS(L)$ such that
\[
\omega_{\psi_{\RF}, \mu}(g) = \xi_\mu(g) \beta^L(g) M^L_g,\quad g\in \RU(W).
\]
Moreover, 
\[
\omega_{\psi_{\RF}, \mu}(k) = \xi_\mu'(\det k) \cdot M^L_k,\quad k\in \RU(W)_L.
\]
\end{prop}

To ease the notation, we also write $(\omega_{\psi_\RF, \mu}, \CS(L))$ and $(\bar{\omega}_{\psi_\RF}, S)$ as representations of $\RU(W)\ltimes \BH_V$ and $\RU(W)_L\ltimes \BH_L$ respectively. Recall that 
$\BH_V$ and $\BH_L$ act by $\rho_{\psi_\RF}$ and $\bar\rho_{\psi_\RF}$ respectively. 

\begin{cor} \label{weil-type}
There is a nonzero map 
\[
\phi\in \Hom_{\RU(W)_L\ltimes \BH_L}(\omega_{\psi_\RF, \mu}, \bar{\omega}_{\psi_\RF}\otimes \xi_\mu),\quad \phi(f) = f(0), \quad f\in\CS(L).
\]
\end{cor}

\begin{proof}
This can be verified directly using Proposition \ref{gen-latt}, \eqref{rho} and  the action  \eqref{ML} of $M^L_k$, $k\in\RU(W)_L$.
\end{proof}

Put $\RE_L :=\det(\RU(W)_L)$. Then we have that 
\[
\RE_L = \RE^1:=\Set{t\in \RE^\times | N_{\RE/\RF}(t)=1}
\]
if $\RE/\RF$ is unramified, and 
\[
\RE_L = \RE^+ : = \RE^1\cap (1+\Fp_\RE)
\]
is an index 2 subgroup of $\RE^1$ otherwise. The restriction $\xi_\mu'|_{\RE^+}$ is determined in \cite{O}, based on the calculations in \cite{P1}. In particular, $\xi_\mu' : \RE_L\to\BC^\times$ is determined if $\RE/\RF$ is ramified. 

Let us sketch the  description of $\xi_\mu'|_{\RE^+}$. By Hilbert's 90, there is a commutative diagram 
\[
\xymatrix{
(1+\Fp_\RE)/(1+\Fp_\RF) \ar@{^(->}[r] \ar[d]^\cong & \CO_\RE^\times/\CO_\RF^\times \ar@{^(->}[r] \ar[d]^\cong &  \RE^\times /\RF^\times \ar[d]^\cong \\
\RE^+ \ar@{^(->}[r] & \RE_L \ar@{^(->}[r] & \RE^1 
}
\]
where the vertical isomorphisms are induced by the map $k\mapsto k/\iota(k)$. 
Since $\mu$ is trivial on $N_{\RE/\RF}(\RE^\times)$,  $\mu|_{1+\Fp_\RE}$ factors through the quotient $(1+\Fp_\RE)/(1+\Fp_\RF)$ and can be regarded as a character of $\RE^+$. Denote this character by $\mu^+$.
By \cite[Proposition 3.1]{O},
\begin{equation}
\xi_\mu' |_{\RE^+} = \mu^+.
\end{equation}

\section{Minimal $K$-types of induced representations} \label{KTIR}

Let $V=\RE^{2m}$ be endowed with a skew-Hermitian form represented by $J_{2m}$. Denote by \[
\{e_1,\ldots, e_m, f_m,\ldots, f_1\}
\]
the standard basis 
of $V$ so that 
\[
\langle e_i, e_j\rangle = \langle f_i, f_j\rangle=0,\quad \langle e_i, f_j\rangle = \delta_{ij}.
\] 
As before we assume that $m \geq 2$, and put $n=\lfloor\frac{m+1}{2}\rfloor$. Let
\[
V= X+ W + X^\vee
\]
be a polarization, where $X= {\rm Span}_{\RE}\{e_1,\ldots, e_{l_0}\}$ and $X^\vee={\rm Span}_{\RE}\{f_1,\ldots, f_{l_0}\}$ are isotropic subspaces of dimension $l_0 = m-n$ and are dual to each other.  Set
\begin{align*}
& G=\RU(V) = \RU_{2m}(\RF), \\
& H=\RU(W) = \RU_{2n}(\RF), \\
& M=\GL(Y) = \GL_m(\RE),
\end{align*}
where $Y={\rm Span}_{\RE}\{e_1,\ldots, e_m\}$. 
Following \cite{BS}, we identify a point in the building of a classical group with a lattice function. 
Assume that $x\in \CB(M)\subset \CB(G)$ is the hyperspecial point which corresponds to the self-dual lattice 
\[
V_{x, 0} = \CO_\RE^{2m},
\]
so that
\[
V_{x, 0}\cap W = \CO_\RE^{2n} = L
\]
is the lattice of $W$ as in Section \ref{sec-MSWR}. Let $y\in \CB(H)$ be the  hyperspecial point corresponding to $L$. Then 
\[
W_{y, 0}=L,\quad H_y = \RU(W)_L, \quad H_y / H_{y, 0+} = {\sf H}^F = {\sf U}_{2n}(\Ff)  \  \textrm{or}  \ {\sf Sp}_{2n}(\Ff)
\]
as in \eqref{fH}. 
Recall that $P$ is the Siegel parabolic subgroup of $G$ with Levi subgroup $M$. 
We also have the quotient
\[
 {\sf G}^F = G_x/G_{x, 0+} = {\sf U}_{2m}(\Ff) \quad \textrm{or} \quad {\sf Sp}_{2m}(\Ff)
\]
according to whether $\RE/\RF$ is unramified or not, which has the Siegel parabolic subgroup ${\sf P}^F = P_x/ P_{x, 0+}$ and Levi subgroup 
\[
{\sf M}^F= M_x/ M_{x, 0+} = \GL_m(\Fe).
\]
Let $\tau = \textrm{c-}\Ind^M_{M_{[x]}}\bar\tau$ be an irreducible distinguished depth zero supercuspidal representation of $M$, where $\bar\tau = \bar\tau_{s, c_s}$ as given by Theorem \ref{thm-dist}. Then 
$\bar\tau|_{M_x}$ is inflated from an
 irreducible (conjugate) self-dual cuspidal representation 
 \[
 \bar\tau' :=(-1)^{m+1}R^{\sfM}_{\sfT, s}
 \]
 of ${\sf M}^F$. We have
 \[
 \Ind^{{\sf G}}_{{\sf P}} (\bar\tau') =  \pi_{\reg} \oplus \pi_{\rm ss}
 \]
 as in \eqref{Dec} and \eqref{Dec-SP}.  Recall the parabolic induction $\rho_{\tau, {\sfs}} = \Ind^G_P (\tau_{\sfs})$, ${\sfs}\in\BC$, given by \eqref{PI}, and the short exact sequence
 \eqref{SES} at ${\sfs}=1$. For convenience, write
 \[
 {\sfs}' = {\sfs}+m/2. 
 \]
 
By abuse of notation, in the rest of the paper we often identify a representation of a finite quotient with its inflation.  The main result of this section is the following.
 
 \begin{prop} \label{K-type}
 With above definitions, we have that
 \begin{enumerate}
 \item  $\rho_{\tau, {\sfs}}^{G_{x, 0+}} \cong \pi_\reg \oplus \pi_{\rm ss}$,  ${\sfs}\in\BC$,
 \item
$\rho_{\tau}^{G_{x,0+}} \cong \pi_\reg$, $\pi_\tau^{G_{x, 0+}} \cong \pi_{\rm ss}$.
 \end{enumerate}
 \end{prop}

\begin{proof} (1)
Let  $N$ be the unipotent radical of $P$, and let ${\sf N}^F=N_x/N_{x,0+}$.
Take $z\in \CB(M)$  such that the following hold:
\begin{itemize}
\item $G_{x,0+}\subset G_{z, 0+}\subset G_z\subset G_x$, 

\item $G_{z,0+}/G_{x,0+}\cong {\sf N}^F \subset {\sf G}^F$,

\item $G_z/G_{z,0+}\cong M_z/M_{z,0+}$.
\end{itemize}
Define the following representation of $M_{[x]}=ZM_x$,
\begin{equation} \label{taubar}
\bar{\tau}_{\sfs}:=\bar{\tau}\otimes |\det|_\RE^{{\sfs}-1/2},\quad {\sfs}\in\BC,
\end{equation}
so that
\[
\tau_{\sfs}=\textrm{c-Ind}^M_{M_{[x]}}\bar{\tau}_{\sfs}.
\]
Denote by $\CJ_N$ the Jacquet functor with respect to $N$. By \cite[Proposition 6.7]{MP}, 
\begin{equation} \label{z+}
\rho_{\tau, {\sfs}}^{G_{z, 0+}} \cong \CJ_N(\rho_{\tau, {\sfs}})^{M_{z, 0+}} \cong \bar\tau' \oplus \bar\tau'.
\end{equation}
On the other hand,
\begin{equation} \label{x+}
\begin{aligned}
\rho_{\tau, {\sfs}}^{G_{x, 0+}} & =\left( \textrm{c-}\Ind^G_{ZM_zN}(\bar\tau_{{\sfs}'})\right)^{G_{x, 0+}}\\
&\supset \left( \textrm{c-}\Ind^{G_x}_{ZM_zN}(\bar\tau_{{\sfs}'})\right)^{G_{x, 0+}} \\
& \cong \Ind^{\sf G}_{\sf P}(\bar\tau' )  \\
& = \pi_\reg\oplus \pi_{\rm ss}.
\end{aligned}
\end{equation}
It is clear that
\[
\pi_\reg^{{\sf N}^F} \cong \pi_{\rm ss}^{{\sf N}^F}\cong\bar\tau' 
\]
as ${\sf M}^F$-modules. By the uniqueness of minimal $K$-types, every ${\sf G}^F$-modules in $\rho_{\tau, {\sfs}}^{G_{x,0+}}$ has cuspidal support $\bar\tau'$.  Thus it follows from \eqref{z+} 
and \eqref{x+} that equality must hold in \eqref{x+}, that is, 
\[
\rho_{\tau, {\sfs}}^{G_{x, 0+}}\cong \pi_\reg \oplus \pi_{\rm ss}. 
\]
(2) We first note that
\[
\CJ_N(\rho_\tau) \cong \CJ_N(\pi_\tau) \cong \tau
\]
and therefore
\[
\CJ_N(\rho_\tau)^{M_{x, 0+}}\cong \CJ_N(\pi_\tau)^{M_{x, 0+}}\cong\bar\tau',
\]
It follows that exactly one of $\rho_\tau^{G_{x, 0+}}$ and $\pi_\tau^{G_{x, 0+}}$ is isomorphic to $\pi_\reg$, and the other is isomorphic to $\pi_{\rm ss}$.

Let $U$ be the unipotent radical of the standard Borel subgroup of $G$, and let ${\sf U}^F = U_x/U_{x, 0+}$. Let $\psi$ be a generic character of $U$ of depth zero.  By Frobenius reciprocity, we have
\begin{align*}
0 &\neq \Hom_G(\rho_{\tau,{\sfs}}, \Ind^G_U\psi) \\
& = \Hom_G(\textrm{c-}\Ind^G_{ZM_zN }\bar\tau_{{\sfs}'}, \Ind^G_U\psi) \\
& \supset \Hom_{M_zN}(\bar\tau_{{\sfs}'}, \Ind^{M_zN}_{M_zN\cap U}\psi) \\
& = \Hom_{M_z} (\bar\tau', \Ind^{M_z}_{M_z\cap U}\psi) \\ 
& \supset \Hom_{M_zN_zG_{z, 0+}}(\bar\tau', \Ind^{G_z}_{G_z\cap U}\psi) \\
& = \Hom_{{\sf U}^F}(\Ind^{\sf G}_{\sf P}(\bar\tau'), \psi|_{{\sf U}^F}).
\end{align*}
By the uniqueness of Whittaker models, all the above spaces are one-dimensional and the inclusions are in fact equalities. Recall that $\pi_\reg$ is generic and $\pi_{\rm ss}$ is not. It follows that 
a nonzero Whittaker functional of $\rho_{\tau, {\sfs}}$ is nonvanishing on $\pi_\reg \subset \rho_{\tau, {\sfs}}^{G_{x, 0+}}$. Following \cite{Ki},  there is an intertwining operator $\rho_{\tau, 1}\to \rho_{\tau, -1}$, with image isomorphic to $\pi_\tau$.  Hence $\pi_\reg\not\subset \pi_\tau^{G_{x, 0+}}$, which implies that
\[
\pi_\reg\cong \rho_\tau^{G_{x,0+}},\quad \pi_{\rm ss} \cong \pi_\tau^{G_{x, 0+}}.
\]
\end{proof}

\section{Explicit depth zero local descent} \label{sec-EDZLD}

We keep the notations from the previous sections. We first recap the irreducible supercuspidal representations $\sigma_{s, a}$ and $\sigma$ of $H= \RU(W) =  \RU_{2n}(\RF)$ given by \eqref{siga} and \eqref{sig} respectively.  Recall the representation $\tau_s$ of $\GL_m(\RE)$ given by Theorem \ref{thm-dist}, and the character $\xi_\mu$ of $H_y = \RU(W)_L$ defined in Section \ref{sec-MSWR}.

\begin{itemize}
\item $\RE/\RF$ is unramified, $m=2n-1$. In this case, as in the proof of Theorem \ref{U-des}, for each $a\in \Ff_2^1$ define
\[
\barsigma_{s,a} = (-1)^n R^{{\sf U}_{2n}}_{\sfS_0, (-s,a)},
\]
where $\sfS_0\in\CT({\sf U}_{2n})$, $\sfS_0^F\cong \Ff_{2n-2}^1\times \Ff_2^1$ is as in Theorem \ref{U-des} (2). Define
\[
\sigma_{s,a} = \textrm{c-}\Ind^H_{H_y}(\barsigma_{s,a}\otimes \xi_\mu^{-1}). 
\]
\item $\RE/\RF$ is ramified, $m=2n$. In this case, define
\[
\bar\sigma_s= (-1)^n R^{{\sf Sp}_{2n}}_{\sfS_0, -s},
\]
where $\sfS_0\in\CT({\sf Sp}_{2n})$, $\sfS_0^F\cong \Ff_{2n}^1$ is as in Theorem \ref{SP-des} (2). Define
\[
\sigma_s = \textrm{c-}\Ind^H_{H_y}(\bar\sigma_s\otimes \xi_\mu^{-1}). 
\]
Recall that in this case $\xi_\mu = \mu^+\circ \det$, where $\mu^+$ is the character of $\RE^+$ defined at the end of Section \ref{sec-MSWR}.
\end{itemize}
Now we state and prove Theorem \ref{thm:DD} as follows, which is the main result on the depth zero local descent for unitary groups. 

\begin{thm} \label{thm:LD} Let the notations be as above. In particular $\tau :=\tau_s$ is an irreducible distinguished depth zero  supercuspidal representations of $\GL_m(\RE)$ as in Theorem \ref{thm-dist}, with $m\geq 2$. Then the following hold.
\begin{enumerate}
\item If $\RE/\RF$ is unramified, $m=2n-1$, then  $\sigma_{s, a}$ is a multiplicity free direct summand of $\CD_{n-1,\psi_\RF,\mu}(\pi_\tau)$ for each $a\in \Ff_2^1$.
\item If $\RE/\RF$ is ramified, $m=2n$, then
\[
\CD_{n,\psi_\RF,\mu}(\pi_\tau) = \sigma_s. 
\]
\end{enumerate}
\end{thm}

\begin{proof}
We  shall only give the proof of (1), and the proof of (2) is similar. By the uniqueness of Fourier-Jacobi models or the more specific Theorem \ref{ST-thm}, it suffices to show that 
\begin{equation} \label{main-id}
\Hom_H(\CD_{n-1,\psi_\RF,\mu}(\pi_\tau), \sigma_{s,a}) = \Hom_{R_{n-1}}(\pi_\tau, \sigma_{s,a}\otimes\nu_{n-1,\psi_\RF, \mu})\neq 0.
\end{equation} 
We refer to Section \ref{ssec-LDGGP}  for the definition of the subgroup $R_{n-1}=H\ltimes N_{n-1}$ of $G$ and its representation 
\[
\nu_{n-1, \psi_\RF, \mu} = \psi_{n-1}\otimes \omega_{2n, \psi_\RF, \mu}.
\]
In particular, $\omega_{\psi_\RF, \mu} = \omega_{2n, \psi_\RF, \mu}$ is the Weil representation of $H\ltimes \BH_W$, and the character $\psi_{n-1}$ of $Z_{n-1}\subset \GL_{n-1}(\RE)$ is of depth zero defined using $\psi_\RF$.

Following \cite{ST}, we have the  open  $(P, R_{n-1})$-double coset in $G$,
\[
\CO := P\gamma R_{n-1} = P \gamma H N_{n-1},
\]
where
\[
\gamma := \begin{pmatrix} 
0 & I_n & 0 &  0 \\
0 & 0 & 0 & -I_{n-1} \\
I_{n-1} & 0 & 0 & 0  \\
0 & 0 & I_n & 0
\end{pmatrix}.
\]
Define a subspace of $\rho_{\tau, \sfs}$,
\[
\CS(\CO, \tau_{\sfs}) := \Set{ f\in \rho_{\tau, \sfs} | {\rm supp}(f)\subset\CO}.
\]
By \cite[Proposition 2.1]{ST}, we have a natural isomorphism 
\begin{equation} \label{iso}
\CD_{n-1,\psi_\RF,\mu}(\rho_{\tau,\sfs}) \cong \CD_{n-1,\psi_\RF, \mu}(\CS(\CO, \tau_{\sfs}))
\end{equation} 
of $H$-modules. 

Put $P^\gamma := \gamma^{-1}P\gamma \cap R_{n-1}$, so that there is a natural bijection 
\[
P\backslash P\gamma R_{n-1} \to P^\gamma\backslash  R_{n-1},\quad P\gamma g \mapsto P^\gamma g,\quad g\in  R_{n-1}.
\]
It follows that
\begin{equation} \label{SO}
\CS(\CO,\tau_{\sfs}) \cong \textrm{c-}\Ind^{R_{n-1}}_{P^\gamma} \tau_{\sfs'}^\gamma,
\end{equation} 
where $\tau_{\sfs'}^\gamma(p) = \tau_{\sfs'}(\gamma p\gamma^{-1})$, $p\in P^\gamma$.  Put
$
P_{[x]}= M_{[x]}N,
$
and define the open subgroup of $P^\gamma$,
\[
P^\gamma_{[x]}:= \gamma^{-1} P_{[x]}\gamma \cap H_y N_{n-1}. 
\]
Direct calculation shows that $P^\gamma_{[x]}$ is compact, with finite quotient ${(\sf P}^\gamma)^F$, where
\[
{\sf P}^\gamma := \gamma^{-1}{\sf P}\gamma \cap {\sf HN}. 
\]
Recall that $\tau_{\sfs} = \textrm{c-Ind}^M_{M_{[x]}}\bar{\tau}_{\sfs}$, where $\bar\tau_{\sfs}$ is given by 
\eqref{taubar}. Applying \eqref{iso}, \eqref{SO} and Frobenius reciprocity,  
\begin{equation} \label{main-int}
\begin{aligned}
& \Hom_H ( \CD_{n-1, \psi_\RF, \mu}(\rho_{\tau, {\sfs}}), \sigma_{s,a}) \\
\supset \ & \Hom_{H_y}\left(\CJ_{N_{n-1}} \left( \textrm{c-}\Ind^{H_y N_{n-1}}_{ P_{[x]}^\gamma} (\bar\tau^\gamma_{{\sfs}'} )\otimes  \nu_{n-1, \psi_\RF, \mu}^\vee\right),  \bar\sigma_{s,a} \otimes \xi_\mu^{-1}\right) \\
 = \ & \Hom_{H_y N_{n-1}}\left(\textrm{c-}\Ind^{H_y N_{n-1}}_{ P_{[x]}^\gamma} (\bar\tau^\gamma_{{\sfs}'} )\otimes  \nu_{n-1, \psi_\RF, \mu}^\vee,  \bar\sigma_{s,a}\otimes \xi_\mu^{-1}\right)  \\
  = \ & \Hom_{H_y N_{n-1}} \left(\textrm{c-}\Ind^{H_y N_{n-1}}_{ P_{[x]}^\gamma} (\bar\tau^\gamma_{{\sfs}'} )\otimes  \psi_{n-1}^{-1}\otimes  \omega_{\psi_\RF, \mu}^\vee,  \bar\sigma_{s,a}\otimes \xi_\mu^{-1}\right) \\
  \supset \ & \Hom_{P_{[x]}^\gamma} (\bar\tau_{{\sfs}'}^\gamma \otimes \psi_{n-1}^{-1} \otimes \bar\omega_{\psi_\RF}^\vee \otimes \xi_\mu^{-1}, \bar\sigma_{s,a}\otimes \xi_\mu^{-1}) \\
  = \ &  \Hom_{P_{[x]}^\gamma} (\bar\tau_{{\sfs}'}^\gamma \otimes \psi_{n-1}^{-1} \otimes \bar\omega_{\psi_\RF}^\vee, \bar\sigma_{s,a}),
\end{aligned}
\end{equation}
where the last inclusion is induced by the homomorphism $\phi$ in Corollary \ref{weil-type}. 

The verbatim adaptation of the calculation in \cite{ST} gives the finite field analog of \eqref{main-int} that
\begin{equation}  \label{ST-fin}
\Hom_{{\sf H}^F}\left(\CD_{n-1, \psi_\Ff}(\Ind^{\sf G}_{\sf P}(\bar\tau')), \bar\sigma_{s,a}\right) = \Hom_{({\sf P}^\gamma)^F} \left((\bar\tau')^\gamma \otimes \bar\psi_{n-1}^{-1} \otimes  \omega_{\psi_\Ff}^\vee, \bar\sigma_{s,a}\right),
\end{equation} 
where $\bar\psi_{n-1}$ denotes the reduction of $\psi_{n-1}$ to ${\sf Z}_{n-1}^F$ (cf. Section \ref{ssec-UG}). 
By Theorem \ref{U-des} (1),
\[
\Hom_{{\sf H}^F}\left(\CD_{n-1, \psi_\Ff}(\Ind^{\sf G}_{\sf P}(\bar\tau')), \bar\sigma_{s,a}\right)  = \Hom_{{\sf R}_{n-1}^F}\left( \Ind^{\sf G}_{\sf P}(\bar\tau'),  \bar\sigma_{s,a}\otimes \nu_{n-1, \psi_\Ff}\right)
\]
is one-dimensional, and we take a generator $\bar\varphi$.  Then $\bar\varphi$ through \eqref{ST-fin} inflates to a nonzero map in the last space in \eqref{main-int}, which thereby induces a nonzero map
\[
\varphi \in \Hom_H ( \CD_{n-1, \psi_\RF, \mu}(\rho_{\tau, {\sfs}}), \sigma_{s,a}) = \Hom_{R_{n-1}} (\rho_{\tau, {\sfs}},\sigma_{s,a}\otimes \nu_{n-1,\psi_\RF, \mu}).
\]
Recall from Proposition \ref{K-type} that 
\[
\rho_{\tau, {\sfs}}^{G_{x, 0+}}\cong  \Ind^{\sf G}_{\sf P}(\bar\tau') \cong \pi_\reg \oplus \pi_{\rm ss},\quad \pi_\tau^{G_{x,0+}}\cong \pi_{\rm ss}.
\]  
By unfolding the Frobenius reciprocity, we see that 
$
\varphi\vert_{\rho_{\tau, {\sfs}}^{G_{x, 0+}}}
$
is the inflation of $\bar\varphi$. Theorem \ref{U-des} (2) implies that $\bar\varphi\vert_{\pi_{\rm ss}}\neq 0$, and therefore 
\[
\varphi|_{\pi_{\rm ss}}\neq 0.
\]
Taking ${\sfs}=-1$, it follows that the composition 
\[
\pi_\tau \hookrightarrow  \rho_{\tau, -1} \xrightarrow{\varphi}  \sigma_{s,a}\otimes \nu_{n-1,\psi_\RF, \mu}
\] 
gives a nonzero map in $\Hom_{R_{n-1}}(\pi_\tau, \sigma_{s,a}\otimes \nu_{n-1,\psi_\RF, \mu})$, hence \eqref{main-id} holds.
\end{proof}


\begin{thebibliography}{000000}



\bibitem[AKM+]{AKM+}
U. K. Anandavardhanan, R. Kurinczuk,  N. Matringe, V. S\'echerre, S. Stevens, {\it Galois self-dual cuspidal types and Asai local factors,} J. Eur. Math. Soc. (JEMS) {\bf 23} (2021), no. 9, 3129--3191.

\bibitem[AKT]{AKT}
U. K. Anandavardhanan,   A. C. Kable,  R. Tandon, {\it Distinguished representations and poles of twisted tensor L-functions,}
Proc. Amer. Math. Soc. {\bf 132} (2004), no. 10, 2875--2883.

\bibitem[AMR]{AMR}
A.-M. Aubert, J. Michel, R. Rouquier, {\it Correspondance de Howe pour les groupes reductifs sur les corps finis}, Duke Math. J. {\bf83}, 2 (1996), 353--397.



\bibitem[BS]{BS}
P. Broussous, S. Stevens, {\it Buildings of classical groups and centralizers of Lie algebra elements,} J. Lie Theory 19 (2009), no. 1, 55--78.

\bibitem[BH]{BH}
   C. Bushnell, G. Henniart, {\it 
   The essentially tame local Langlands correspondence. I},
   J. Amer. Math. Soc.18(2005), no.3, 685–710.
   
\bibitem[BK]{BK}
C. Bushnell, P. Kutzko, {\it  The admissible dual of $\GL(N)$ via compact open subgroups},
Ann. of Math. Stud., 129 Princeton University Press, Princeton, NJ, 1993. xii+313 pp.


\bibitem[Ch]{Ch}
K. Y. Chan, {\it Restriction for general linear groups: the local non-tempered Gan-Gross-Prasad conjecture (non-Archimedean case),}
J. Reine Angew. Math. {\bf 783} (2022), 49--94.

\bibitem[CG]{CG}
C. Coniglio-Guilloton, {\it Correspondance de Jacquet-Langlands et distinction: cas des repr\'esentations cuspidales de niveau $0$,}  Bull. Soc. Math. France {\bf 144} (2016), no. 2, 163--216.

\bibitem[D]{D}
P. Deligne, {\it La conjecture de Weil. II.} Inst. Hautes \'Etudes Sci. Publ. Math. No.  52 (1980), 137--252.

\bibitem[DL]{DL}
P. Deligne, G. Lusztig, {\it Representations of reductive groups over finite fields}, Ann. of Math. {\bf 103} (1976), 103--161.

\bibitem[F]{F}
Y. Flicker, {\it On distinguished representations,} J. Reine Angew. Math. 418 (1991), 139--172.

\bibitem[GGP1]{GGP1}
W. T. Gan, B. H. Gross, D. Prasad, {\it Symplectic local root numbers, central critical L values, and restriction problems in the representation theory of classical groups}, Sur les conjectures de Gross et Prasad. I. Ast\'erisque {\bf 346} (2012), 1--109.

\bibitem[GGP2]{GGP2}
\bysame, {\it Branching laws for classical groups: the non-tempered case}, Compos. Math. {\bf 156} (2020), no. 11, 2298--2367.

\bibitem[G]{G}
P. G\'erardin, {\it Weil representations associated to finite fields}, J. Algebra {\bf 46} (1977), 54--101.


\bibitem[GH]{GH}
S. Gurevich, R. Hadani, {\it The geometric Weil representation}, Selecta Math. (N.S.) {\bf 13} (2007), no. 3, 465--481. 

\bibitem[GHo]{GHo}
S. Gurevich, R. Howe, {\it Rank and duality in representation theory,} Jpn. J. Math. {\bf 15} (2020), no. 2, 223--309. 


\bibitem[H1]{H1}
J. Hakim, {\it Distinguished cuspidal representations over $p$-adic and finite fields,} Pacific J. Math. {\bf 311} (2021), no. 1, 89--111.

\bibitem[H2]{H2}
\bysame, {\it Distinguished Regular Supercuspidal Representations and Inductive Constructions of Representations,} \href{https://arxiv.org/abs/1808.03982}{arXiv:1808.03982}. 

\bibitem[HM1]{HM1}
J. Hakim, F.  Murnaghan, {\it Two types of distinguished supercuspidal representations,} Int. Math. Res. Not. 2002, no. 35, 1857--1889.

\bibitem[HM2]{HM2}
\bysame, {\it Distinguished tame supercuspidal representations,} Int. Math. Res. Pap. IMRP 2008, no. 2, Art. ID rpn005, 166 pp.




\bibitem[HS]{HS}
G. Hiss, M. Schr\"oer, {\it Two conjectures on the Weil representations of finite symplectic and unitary groups,} J. Algebra {\bf 558} (2020), 485--490.

\bibitem[HZ]{HZ}
G. Hiss, A. Zalesski, {\it The Weil-Steinberg character of finite classical groups},
with an appendix by Olivier Brunat,
Represent. Theory {\bf 13} (2009), 427--459.

\bibitem[Ho1]{Ho1}
R. Howe, {\it Invariant theory and duality for classical groups over finite fields with applications to their
singular representation theory}, unpublished paper, available at \url{https://blog.nus.edu.sg/matzhucb/links/}

\bibitem[Ho2]{Ho2}
\bysame, {\it On the character of Weil's representation}, Trans. Amer. Math. Soc. {\bf 177} (1973), 287--298.


\bibitem[JNQ]{JNQ}
D. Jiang, C. Nien, Y. Qin, {\it Symplectic supercuspidal representations of ${\rm{GL}}(2n)$ over $p$-adic fields}, Pacific J. Math. {\bf 245} (2010), no. 2, 273--313.

\bibitem[JS]{JS}
D. Jiang, D. Soudry, {\it Appendix: On the local descent from ${\rm {GL}}(n)$ to classical groups} [appendix to MR2931222],
Amer. J. Math. {\bf 134} (2012), no. 3, 767--772.

\bibitem[JZ]{JZ}D. Jiang, L. Zhang, {\it Local root numbers and spectrum of the local descents for orthogonal groups: $p$-adic case}, Algebra Number Theory {\bf 12} (2018), no. 6, 1489--1535.

\bibitem[K]{K}
A. C. Kable, {\it 
Asai $L$-functions and Jacquet's conjecture,}
Amer. J. Math. {\bf 126} (2004), no. 4, 789--820.

\bibitem[Ka]{Ka}
T. Kaletha,
{\it Regular supercuspidal representations,}
J. Amer. Math. Soc. {\bf 32} (2019), no. 4, 1071--1170.

\bibitem[Ki]{Ki}
H. Kim, {\it On local L-functions and normalized intertwining operators,} 
Canad. J. Math. 57 (2005), no. 3, 535--597.

\bibitem[KK]{KK}
H. Kim, M. Krishnamurthy, {\it 
Stable base change lift from unitary groups to $\GL_n$,}
IMRP Int. Math. Res. Pap. 2005, no. 1, 1--52.


\bibitem[Ku]{Ku}
S. Kudla, {\it Splitting metaplectic covers of dual reductive pairs,}
Israel J. Math. 87 (1994), no. 1-3, 361--401.

\bibitem[L1]{L1}
G. Lusztig, {\it Irreducible representations of finite classical groups}, Invent. Math. {\bf 43} (1977), 125--175.

\bibitem[L2]{L2}
\bysame, {\it Characters of reductive groups over a finite field}, Princeton Univ. Press, Princeton, N.J., 1984.

\bibitem[L3]{L3}
\bysame, {\it On the representations of reductive groups with disconnected centre,} in: Orbites unipotentes et
repr\'esentations I. Groupes finis et alg\`ebres de Hecke, Ast\'erisque {\bf 168}, Soci\'et\'e Math\'ematique de France, Paris
(1988), 157--166.

\bibitem[L4]{L4}
\bysame, {\it Symmetric spaces over a finite field,} The Grothendieck Festschrift, Vol. III, 57--81, Progr. Math., 88, Birkh\"auser Boston, Boston, MA, 1990. 

\bibitem[LS]{LS}
G. Lusztig, B. Srinivasan, {\it The characters of the finite unitary groups,}
J. Algebra {\bf 49} (1977), no. 1, 167--171.

\bibitem[LMNW]{LMNW}
D. Liu, J. Ma, C. Nien, Z. Wang, {\it Explicit local descent from ${\rm{GL}}_{2n}$ to ${\rm{SO}}_{2n+1}$: depth zero case}, in preparation. 

\bibitem[LW1]{LW1}
D. Liu, Z. Wang, {\it Descents of unipotent representations of finite unitary groups}, Trans. Amer. Math. Soc. {\bf 373} (2020), no. 6, 4223--4253.

\bibitem[LW2]{LW2} \bysame, {\it Remarks on the theta correspondence over finite fields}, Pacific J. Math. {\bf 306} (2020), no. 2, 587--609.

\bibitem[LW3]{LW3} \bysame, {\it On the Gan-Gross-Prasad problem for finite unitary groups}, Math. Z. {\bf 297} (2021), no. 3-4, 997--1021.

\bibitem[LW4]{LW4} \bysame, {\it Descents of unipotent cuspidal representations of finite classical groups}, Manuscripta Math. {\bf 165} (2021), no. 1-2, 159--189.


\bibitem[MP]{MP}
A. Moy, G. Prasad, {\it Jacquet functors and unrefined minimal K-types,} Comment. Math. Helv. {\bf 71} (1996), no. 1, 98--121. 
\bibitem[MQZ]{MQZ}
J.-J. Ma, C. Qiu, J. Zou, {\it Generic Hecke algebra and theta correspondence over finite fields,}  
\href{https://arxiv.org/abs/2208.00431}{arXiv:2208.00431}.

\bibitem[O]{O}
M. Oi,  {\it Depth preserving property of the local Langlands correspondence for non-quasi-split unitary groups,} Math. Res. Lett. 28 (2021), no. 1, 175--211. 

\bibitem[P1]{P1}
S.-Y. Pan, {\it  Splittings of the metaplectic covers of some reductive dual pairs},  Pacific J. Math. 199 (2001), no. 1, 163--226.



\bibitem[P2]{P2}
\bysame, {\it Howe correspondence of unipotent characters for a finite symplectic/even-orthogonal dual pair}, \href{https://arxiv.org/abs/1901.00623}{arXiv:1901.00623}.

\bibitem[P3]{P3}
\bysame, {\it Lusztig correspondence and Howe correspondence for finite reductive dual pairs}, \href{https://arxiv.org/abs/1906.01158}{arXiv:1906.01158}.


\bibitem[Pr1]{Pr1}
D. Prasad, {\it On a conjecture of Jacquet about distinguished representations of $\GL(n)$,} Duke Math. J. 109 (2001), no. 1, 67--78.

\bibitem[Pr2]{Pr2}
\bysame, {\it Multiplicities under basechange: finite field case,} 
J. Algebra {\bf 556} (2020), 1101--1114.

\bibitem[RR]{RR}
R.  Ranga-Rao, {\it On some explicit formulas in the theory of Weil representation,} Pacific J. Math. 157 (1993), no. 2, 335--371.

\bibitem[R]{R}
M. Reeder, {\it On the restriction of Deligne-Lusztig characters},
J. Amer. Math. Soc. {\bf 20} (2007), no. 2, 573--602.

\bibitem[SGA5]{SGA5}
{\it Cohomologie l-adique et fonctions L}, Lecture Notes in Mathematics, Vol. 589 (Springer-Verlag, Berlin-New York, 1977), S\'eminaire de G\'eometrie Algébrique du Bois-Marie 1965--1966 (SGA 5), Edit\'e par Luc Illusie. 

\bibitem[SV]{SV}
Y. Sakellaridis,  A. Venkatesh, 
{\it Periods and harmonic analysis on spherical varieties}, Ast\'erisque (2017), no.396, viii+360 pp.

\bibitem[S]{S}
V. S\'echerre, {\it Supercuspidal representations of $\GL_n(F)$ distinguished by a Galois involution,} Algebra Number Theory {\bf 13} (2019), no. 7, 1677--1733.

\bibitem[Shi]{Shi}
F. Shi, {\it Periods of Deligne-Lusztig characters associated to spherical varieties,} in preparation.

\bibitem[ST]{ST}
D. Soudry, Y. Tanay, {\it On local descent for unitary groups}, J. Number Theory {\bf 146} (2015), 557--626.

\bibitem[Sun]{Sun}
B. Sun, {\it Multiplicity one theorems for Fourier-Jacobi models}, Amer. J. Math. {\bf 134} (2012), no.6, 1655--1678.

 

\bibitem[T]{T}
T. Thomas, {\it The character of the Weil representation}, J. Lond. Math. Soc. (2) {\bf 77} (2008), no. 1, 221--239. 

\bibitem[W]{W}
Z. Wang, {\it On the Gan-Gross-Prasad problem for finite classical groups,} Adv. Math. {\bf 393} (2021), Paper No. 108095, 72 pp.

\bibitem[Z]{Z}
C. Zhang, {\it 
Distinguished regular supercuspidal representations,}
Math. Ann. {\bf 376} (2020), no. 3-4, 1561--1598.

\end{thebibliography}
\end{document}